\documentclass[12pt]{amsart}
\usepackage{amsfonts,amssymb,amscd,xy,bm}
\usepackage[leqno]{amsmath}
\usepackage{mathrsfs}
\usepackage{amssymb}
\usepackage{amsmath}
\usepackage{amsxtra}
\usepackage{amsthm}
\usepackage[backref]{hyperref} 
\numberwithin{equation}{section}

\usepackage[a4paper, total={5.5in, 7in}]{geometry}

\def\si{\Sigma}

\def\ofs{\mathcal O_{F\!,\e \si}}

\def\ofss{\mathcal O_{\!F\!,\e \si}^{\e *}}

\def\oks{\mathcal O_{\be K,\e \si}}

\def\okss{\mathcal O_{\! K\!,\e \si}^{\e *}}

\def\pic{{\rm{Pic}}\,}

\def\br{{\rm{Br}}\e}
\newcommand\brp{{\rm Br}^{\le\prime}}

\newcommand{\isoto}{\overset{\!\be\sim}{\to}}

\newcommand{\et}{{\rm {\acute et}}}

\DeclareMathAlphabet{\mathbbmsl}{U}{bbm}{m}{sl}

\newcommand{\fl}{{\rm fl}}
\newcommand{\spp}{S^{\le\lle\prime}}

\def\ses{S_{\et}^{\e\sim}}

\def\sfs{S_{\fl}^{\e\sim}}

\def\bg{{\mathbb G}}

\def\hom{{\rm{Hom}}\e}

\usepackage[usernames,dvipsnames]{xcolor}
\usepackage{tabularx}
\usepackage{epsfig,amssymb}
\usepackage{bm} 
\usepackage{verbatim} 

\usepackage{hyperref}
\definecolor{labelkey}{rgb}{1,0,0}

\newcommand{\spec}{\mathrm{ Spec}\,}

\newcommand{\s}{\mathscr }

\def\ofs{\mathcal O_{\be F\lbe,\e \si}}

\def\oks{\mathcal O_{\lbe K,\e \si}}

\def\e{\kern 0.06em}
\def\be{\kern -.1em}
\def\le{\kern 0.03em}
\def\lle{\kern 0.015em}
\def\lbe{\kern -.025em}
\def\Z{\mathbb Z}

\newcommand{\sh}{\kern -.4em\phantom{a}^{\mathbf{\sim}}}

\newcommand{\lra}{\longrightarrow}

\def\be{\kern -.1em}
\def\le{\kern 0.03em}
\def\lle{\kern 0.04em}
\def\lbe{\kern -.025em}

\newcommand{\Q}{{\mathbb Q}}

\newcommand{\krn}{\mathrm{Ker}\e}
\newcommand{\cok}{\mathrm{Coker}\e}

\def\e{\kern 0.08em}

\newcommand{\img}{\mathrm{Im}\e}

\DeclareFontEncoding{OT2}{}{} 
\newcommand{\textcyr}[1]{%
{\fontencoding{OT2}\fontfamily{wncyr}\fontseries{m}\fontshape{n}
\selectfont #1}}
\newcommand{\sha}{{\mbox{\textcyr{Sh}}}}

\newtheorem{lemma}{Lemma}[section]
\newtheorem{theorem}[lemma]{Theorem}
\newtheorem{proposition-definition}[lemma]{Proposition-Definition}
\newtheorem{corollary}[lemma]{Corollary}
\newtheorem{proposition}[lemma]{Proposition}
\theoremstyle{definition}
\newtheorem{definition}[lemma]{Definition}

\theoremstyle{remark}
\newtheorem{remark}[lemma]{Remark}
\newtheorem{remarks}[lemma]{Remarks}

\newtheorem{examples}[lemma]{Examples}

\begin{document}

\input xy     
\xyoption{all}

\title[Restriction cokernels and norm one subgroups]{Cokernels of restriction maps and subgroups of norm one, with applications to quadratic Galois coverings}

	
\author{Cristian D. Gonz\'alez-Avil\'es}
\address{Departamento de Matem\'aticas, Universidad de La Serena,
La Serena, Chile} \email{cgonzalez@userena.cl}

	
\thanks{The author was partially supported by Fondecyt grant
1160004}

\maketitle

\begin{abstract} Let $f\colon \spp\to S$ be a finite and faithfully flat morphism of locally noetherian schemes of constant rank $n$ and let $G$ be a smooth, commutative and quasi-projective $S$-group scheme with connected fibers. For every $r\geq 1$, let ${\rm Res}_{\le G}^{(r)}\colon H^{\le r}\be(S_{\et},G\e)\to H^{\le r}\be(\spp_{\et},G\e)$ and ${\rm Cores}_{\le G}^{(r)}\colon H^{\le r}\be(\spp_{\et},G\e)\to H^{\le r}\be(S_{\et},G\e)$ be, respectively, the restriction and corestriction maps in \'etale cohomology induced by $f$. For certain pairs $(\le f, G\le)$, we construct maps $\alpha_{\lle r}\colon \krn {\rm Cores}_{\le G}^{(r)}\rightarrow \cok\e {\rm Res}_{\le G}^{(r)}$ and $\beta_{\lle r}\colon \cok\e {\rm Res}_{\le G}^{(r)}\rightarrow \krn {\rm Cores}_{\le G}^{(r)}$ such that $\alpha_{\lle r}\circ \beta_{\lle r}=\beta_{\lle r}\circ\e\alpha_{\lle r}=n$. In the simplest nontrivial case, namely when $f$ is a quadratic Galois covering, we identify the kernel and cokernel of $\beta_{\lle r}$ with the kernel and cokernel of another map $\cok\e {\rm Cores}_{\e G}^{(r-1)}\to\krn {\rm Res}_{G}^{(r+1)}$. We then discuss several applications, for example to the problem of comparing the (cohomological) Brauer group of a scheme $S$ to that of a quadratic Galois cover $\spp$ of $S$.
\end{abstract}

\section{Introduction}

Let $f\colon \spp\to S$ be a finite and faithfully flat morphism of locally noetherian schemes of constant rank $n\geq 2$ and let $G$ be a smooth, commutative and quasi-projective $S$-group scheme with connected fibers. In \cite{ga18b} we discussed the kernel of the canonical restriction map in \'etale cohomology ${\rm Res}_{\le G}^{(r)}\colon H^{\le r}\lbe(S_{\et},G\e)\to H^{\le r}\lbe(\spp_{\et},G\e)$ for every integer $r\ge 1$ and a broad class of pairs $(\le f, G\le)$ as above. In the present paper we discuss the cokernel of the preceding map and relate it to the kernel of the corestriction map ${\rm Cores}_{\le G}^{(r)}\colon H^{\le r}\lbe(\spp_{\et},G\e)\to H^{\le r}\lbe(S_{\et},G\e)$ in the same broad setting of \cite{ga18b}. More precisely, we show that the groups just mentioned are canonically related to each other via maps $\alpha_{\le r}\colon \krn {\rm Cores}_{\le G}^{(r)}\to \cok\e {\rm Res}_{\le G}^{(r)}$ and $\beta_{\le r}\colon \cok\e {\rm Res}_{\le G}^{(r)}\to\krn {\rm Cores}_{\le G}^{(r)}$ such that $\alpha_{\le r}\be\circ\be \beta_{\le r}$ and $\beta_{\le r}\be\circ\be\alpha_{\le r}$ are the multiplication by $n$ maps on their respective domains, where $n$ is the (constant) rank of $f$ (see Theorem \ref{th1}). The kernel (respectively, cokernel) of $\alpha_{\le r}$ is a quotient (respectively, subgroup) of $H^{\le r}(S_{\fl}, G_{\lbe n})$ (respectively, $H^{\le r+1}(S_{\fl}, G_{\lbe n})$), where $G_{\lbe n}$ is the $n$-torsion subgroup scheme of $G$ (see Proposition \ref{ka}). On the other hand, the kernel and cokernel of $\beta_{\le r}$ are related to the kernels and cokernels of certain other maps $c_{\le r}\colon \cok\e{\rm Cores}_{\e G}^{(r-1)}\to H^{\le r+1}\be(S_{\fl},G(n))$ and $d_{\le r}\colon H^{\le r}\be(S_{\fl},G(n))\to \krn {\rm Res}_{\e G}^{(r+1)}$, where $G(n)$ is the fppf sheaf on $S$ \eqref{gn} introduced in \cite{ga18b} (see Proposition \ref{ak}). When $f$ is a Galois covering, $G(n)$ is an $(\lbe\spp\!/\be S\lle)$-form of $G_{\be n}^{\e n-2}$ and thus trivial if $f$ is, in addition, quadratic, i.e., $n=2$ above. In this case (which is that discussed in \cite{kps} and \cite{ps}), Proposition \ref{ak} immediately yields the following statement, which is the main theorem of the paper.

\begin{theorem}\label{main0} {\rm ($=$Theorem \ref{main})} Let $(\e f,G\e)$ be an admissible quadratic Galois pair with  Galois group $\Delta=\{1,\tau\}$ (see Definition {\rm \ref{qadm}}). Then, for every integer $r\geq 1$, there exists a canonical exact sequence of abelian groups
\[
\begin{array}{rcl}
H^{\lle r}\be(S_{\et},G\e)\to H^{\le r}\be(\spp_{\et}, G\e)^{\Delta}&\to& \cok\e {\rm Cores}_{\e G}^{(r-1)}\to\krn {\rm Res}_{G}^{(r+1)}\\
&\to& {}_{N}\be H^{\le r}(\spp_{\et},G\e)/(1\!-\!\tau)H^{\le r}(\spp_{\et},G\e)\to 0,
\end{array}
\]
where ${}_{N}\be H^{\le r}\lbe(\spp_{\et},G\e)=\krn[\e{\rm Cores}_{\le G}^{(r)}\colon H^{\le r}\be(\spp_{\et},G\e)\to H^{\le r}\be(S_{\et},G\e)]$.
\end{theorem}

\begin{corollary} {\rm ($=$Theorem \ref{nice})} Let $K\be/\be F$ be a quadratic Galois extension of global fields and let $\si$ be a nonempty finite set of primes of $F$ containing the archimedean primes and the non-archimedean primes that ramify in $K$. Then there exists a canonical exact sequence of abelian groups
\[
0\to W_{\!F,\e \si}/\lbe N_{\be K\be/\lbe F}\e\okss\to \cok\e j_{ K\be/\lbe F,\e\si}\to C_{K/F,\e\si}\to \cok \lambda\to 0,
\]
where $W_{\!F,\e \si}\subseteq\ofss$ is given by \eqref{kg3}, $j_{ K\be/\lbe F,\e\si}$ is the $\si$-capitulation map \eqref{cap}, $C_{K\be/\lbe F,\e\si}$ is the relative $\si$-ideal class group of $K\!/\be F$ and
\[
\lambda\colon\ofss/\lbe N_{\be K\be/\lbe F}\e\okss\to\krn\!\be\left[\,\displaystyle{\bigoplus_{v\in\si}}\e F_{\! v}^{\e *}\be/\be N_{\be K_{\lbe w}\lbe/\lbe F_{\lbe v}}K_{\lbe w}^{*}\overset{\!\Sigma\e{\rm ord}_{v}}{\lra}\Z/\lbe 2\le\Z \right]
\]
is the canonical localization map. In particular, the number of ideal classes in $C_{\lbe K,\e\si}$ that do not lie in the image of $j_{ K\be/\lbe F,\e\si}$ is at least $[\e W_{\!F,\e \si}\colon\! N_{\be K\be/\lbe F}\e\okss\e]$, and the relative $\si$-ideal class number $|C_{K/F,\e\si}|$ is divisible by $2^{\e \rho\le -\le e}$, where the integers  $\rho$ and $e\leq \rho$ are given by \eqref{sr2} and \eqref{isol}, respectively.
\end{corollary}

\begin{corollary} {\rm ($=$Theorem \ref{abv})} Let $F$ be a global field of characteristic different from $2$, $K\be/\be F$ a quadratic Galois extension with Galois group $\Delta=\{1,\tau\}$ and $\si$ a nonempty finite set of primes of $F$ containing the archimedean primes, the non-archimedean primes that ramify in $K$ and all primes $v$ such that ${\rm char}\e k(v)=2$. Set $S=\spec \mathcal O_{\lbe F,\e \Sigma}$ and $\spp =\mathcal O_{\lbe K,\e \Sigma_{K}}$ and let $\mathcal A$ be an abelian scheme over $S$ with generic fiber $A$. Then there exist canonical exact sequences of torsion abelian groups	
\begin{enumerate}
\item[(i)]
\[
\begin{array}{rcl}	
0&\to& H^{-1}\lbe(\Delta, A(K))\to\sha^{1}_{\Sigma}(F,A\e)\to \sha^{1}_{\Sigma}(K,A\e)^{\Delta}\to A(F\e)/N_{\be K\lbe/\lbe F}A(K)\\\\
&\to&\krn[H^{\le 2}(S,\mathcal A)\to H^{\le 2}(\spp_{\et},\mathcal A\e)]\to {}_{N}\sha^{1}_{\Sigma}(K,A\e)/(1\!-\!\tau)\sha^{1}_{\Sigma}(K,A\e)\to 0
\end{array}
\]
and
\item[(ii)]
\[
\begin{array}{rcl}	
H^{\le 2}(S,\mathcal A)&\to& H^{2}(\spp_{\et},\mathcal A\e)^{\Delta}\to
\sha^{1}_{\Sigma}(F,A\e)/N_{\be K\lbe/\lbe F}\le\sha^{1}_{\Sigma}(K,A\e)\\\\
&\to&\prod\pi_{0}(\lbe A(F_{\be v}))^{\e\prime}\to {}_{N}\be H^{\le 2}(\spp_{\et},\mathcal A\e)/(1\!-\!\tau)H^{\le 2}\lbe(\spp_{\et},\mathcal A\e)\to 0
\end{array}
\]
where the product extends over all real primes $v$ of $F$ that ramify in $K$ and, for each such prime $v$, $\pi_{0}(\lbe A(F_{\be v}))^{\e\prime}$ is the group \eqref{pio}.
\end{enumerate}
\end{corollary}

\smallskip

If $S$ is a scheme, let $\brp S=H^{2}(S_{\et},\bg_{m})$ be the (possibly non-torsion) full cohomological Brauer group of $S$.

\begin{corollary}\label{kf} Let $k$ be a quadratically closed field (e.g., an algebraically closed field or a separably closed field of characteristic different from $2$) and let $Y\to X$ be a quadratic Galois covering of proper and geometrically integral $k$-schemes with Galois group $\Delta=\{1,\tau\}$. Then there exists a canonical exact sequence of abelian groups
\[
\begin{array}{rcl}
0&\to& {}_{N}\pic Y/(1-\tau)\pic Y\to \brp X\to (\brp\e Y\le)^{\Delta}\to \pic X/N_{\le Y\be/\lbe X}\pic Y\\
&\to&\krn[H^{\lle 3}\lbe(X_{\et},\bg_{m})\to H^{\le 3}\lbe(Y_{\et}, \bg_{m})]\to {}_{N}\brp\e Y/(1-\tau)\brp\e Y\to 0.
\end{array}
\]
\end{corollary}
\begin{proof} This follows from Corollary \ref{cor0}(i) and Remark \ref{brr}(b).
\end{proof}

\medskip

The reader will notice that spectral sequences are (conspicuously) absent from this paper, as they were from its predecessor \cite{ga18b}. Although it is certainly possible to use spectral sequences to discuss some of the problems considered in this paper, we avoid them entirely in order to make our presentation as elementary as possible. More precisely, we wish to avoid the type of technical complications that can sometimes arise when spectral sequences are used, such as those appearing in \cite[pp.~227-230]{ps} and \cite[Appendix 1.A, pp.~397-408]{cts87}. See Remark \ref{hs}(a) for more comments on this topic.

\medskip

The paper is organized as follows. After the preliminary Section \ref{pre}, we discuss in Section \ref{pn1} the projective and norm one fppf sheaves on $S$ associated to $(\e S^{\lle\prime}\!\lbe/\be S,G\e)$, which are denoted by $P_{\lbe S^{\lle\prime}\be/\lbe S}(G\e)$  and $R_{\e S^{\lle\prime}\!/\lbe S}^{\e(1)}\lbe(G\e)$, respectively. In Section \ref{compm}\e, inspired by the works of Shyr \cite{sh77}, Ono \cite{o85} and Katayama \cite{k86b}, we introduce comparison morphisms $R_{\e S^{\lle\prime}\!/\lbe S}^{\e(1)}\lbe(G\e)\rightleftarrows P_{\lbe S^{\lle\prime}\be/\lbe S}(G\e)$ and show that, if the pair $(\e S^{\lle\prime}\!\lbe/\lbe S,G\e)$ is admissible in the sense of Definition \ref{adm}, then there exist canonical exact sequences of fppf sheaves of abelian groups on $S$
\[
0\to G_{\lbe n}\to R_{\e S^{\lle\prime}\!/\lbe S}^{\e(1)}\lbe(G\e)\to P_{\lbe S^{\lle\prime}\be/\lbe S}(G\e)\to 0
\]
and
\[
0\to G(n)\to P_{\lbe S^{\lle\prime}\be/\lbe S}(G\e)\to  R_{\e S^{\lle\prime}\!/\lbe S}^{\e(1)}\lbe(G\e)\to 0,
\]
where $G(n)=R_{\e S^{\lle\prime}\!/\lbe S}^{\e(1)}\lbe(\lbe G_{\lbe n}\be)/G_{\lbe n}$. The fppf cohomology sequences induced by the above sequences are then suitably combined in (the rather technical) Section \ref{rc} to construct the maps $\alpha_{\lle r}$ and $\beta_{\lle r}$ mentioned above. In this section we also describe the kernels and cokernels of $\alpha_{\lle r}$ and $\beta_{\lle r}$ in terms of the fppf cohomology of the $n$-torsion sheaves $G_{\lbe n}$ and $G(n)$. In Section \ref{quad} we specialize the general discussion of the preceding sections to the case of quadratic Galois coverings and obtain the main theorem of the paper (viz Theorem \ref{main0} above). In section \ref{arith} we discuss applications of the results obtained in section \ref{quad} to ideal class groups of global fields, N\'eron-Raynaud class groups of invertible tori over such fields and Tate-Shafarevich groups of abelian schemes over the corresponding arithmetic rings. In section \ref{bra} we discuss applications of the results of section \ref{quad} to the cohomological Brauer group of a locally noetherian scheme.

\section*{Acknowledgement}

I thank Dino Lorenzini for sending me a copy of \cite{f92} and the referee for suggesting Examples \ref{ref}.

\section{Preliminaries}\label{pre}

If $A$ is an object of a category, $1_{\be A}$ will denote the identity morphism of $A$.

\begin{lemma}\label{ker-cok}  Let  $A\overset{\!f}{\to}B\overset{\!g}{\to}C$ be  morphisms in an abelian category $\s A$. Then there exists a canonical exact sequence in $\s A$
\[
0\to \krn f\to \krn\lbe(\e g\be\circ\be f\e)\to \krn g\to\cok f\to \cok\lbe(\e g\be\circ\be f\e) \to \cok g\to 0.
\]
\end{lemma}
\begin{proof} See, for example, \cite[1.2]{bey}.
\end{proof}

\begin{remark}\label{k1}
The map $\krn g\to\cok f$ in the sequence of the lemma is the composition $\krn g\hookrightarrow B\twoheadrightarrow \cok f$. The remaining maps are the natural ones.
\end{remark}

If $A$ is an object of an abelian category $\s A$ and $n\geq 1$ is an integer, $A_{\e n}$ (respectively, $A/ n\le$) will denote the kernel (respectively, cokernel) of the multiplication by $n$ morphism on $A$.

\smallskip

{\it All schemes appearing in this paper are tacitly assumed to be non-empty.}

\smallskip

If $S$ is a scheme and $\tau\, (=\et$ or $\fl$) denotes either the \'etale or fppf topology on $S$, $S_{\tau}$ will denote the small $\tau$ site over $S$. Thus $S_{\tau}$ is the category of $S$-schemes that are \'etale (respectively, flat) and locally of finite presentation over $S$ equipped with the \'etale (respectively, fppf) topology. We will write $S_{\tau}^{\e\sim}$ for the (abelian) category of sheaves of abelian groups on $S_{\tau}$. If $G$ is a commutative $S$-group scheme, the presheaf represented by $G$ is an object of $S_{\tau}^{\e\sim}$. If $f\colon\spp\to S$ is an fppf covering of $S$, then the map $G(S\e)\hookrightarrow G(\spp)$ induced by $f$ is an injection that will be regarded as an inclusion. The object $G_{n}$ of the abelian category $\sfs$ is represented by the $S$-group scheme $G\times_{n_{\le G},\e G,\e\varepsilon} S$, where $n_{\le G}$ is the $n$-th power morphism on $G$ and $\varepsilon\colon S\to G$ is the unit section of $G$. If $G$ is separated over $S$, then $G_{n}\hookrightarrow G$ is a closed immersion. If, in addition, $G$ is quasi-projective over $S$, thebut we believe that
n $G_{n}$ is  quasi-projective over $S$ as well \cite[Proposition 5.3.4(i)]{ega2}.

\smallskip

If $\s F$ is an abelian sheaf on $S_{\tau}$ and $i\geq 0$ is an integer, $H^{\le i}\lbe(S_{\tau},\s F\le)$ will denote the $i$-th $\tau$ cohomology group of $\s F$. If $\spp\to S$ is a morphism of schemes, we will write $H^{\le r}\be(\spp_{\tau},G\e)$ for $H^{\le r}\be(\spp_{\tau},G_{\be S^{\lle\prime}}\be)$, where $G_{S^{\lle\prime}}=G\times_{S}\spp$. If $G=\bg_{m,\e S}$, the groups $H^{\le r}\be(S_{\et},G\e)$ will be denoted by $H^{\le r}\be(S_{\et},\bg_{m})$. Note that $H^{\le 0}\lbe(S_{\et},\bg_{m})=\varGamma( S,\mathcal O_{\lbe S}\lbe)^{*}$ is the group of global units on $S$. The groups $H^{\le 1}\be(S_{\et},\bg_{m})$ and $\pic S$ will be identified via \cite[Theorem 4.9, p.~124]{mi1}. Further, we will write $\brp S$ for the cohomological Brauer group of $S$, i.e., $\brp S=H^{2}(S_{\et},\bg_{m})$, and  $\br S$ for the Brauer group of equivalence classes of Azumaya algebras on $S$. By \cite[I, \S2]{dix}, there exists a canonical injection of abelian groups $\br S\hookrightarrow \brp S$. The relative (cohomological) Brauer group associated to a morphism of schemes $\spp\to S$ is the group
\begin{equation}\label{relb}
\brp\lbe(\spp\be/\lbe S\e)=\krn\!\be\left[\brp\le S\to \brp\le\spp\e\right],
\end{equation}
where the indicated map is the restriction map ${\rm Res}_{\e\bg_{m,\le S}}^{(2)}$. The subgroup $\br\lbe(\spp\be/\lbe S\e)$ of $\br S$ is defined similarly.

\smallskip

We recall that, if $f\colon \spp\to S$ is an \'etale morphism and $\Delta$ is a finite group of order $n\geq 2$ which acts on $\spp/S$ from the right, then $f$ is called a {\it Galois covering with Galois group $\Delta$} if the canonical map
\[
\coprod_{\e \delta\e\in\e \Delta}\be \spp\to \spp\!\times_{S}\!\spp, (x,\delta\e)\mapsto (x,x\delta\e),
\]
is an isomorphism of $S$-schemes. See  \cite[V, Proposition 2.6 and Definition  2.8]{sga1}.

\smallskip

Let $f\colon S^{\e\prime}\to S$ be a finite and locally free morphism of schemes and let $X^{\prime}$ be an $\spp$-scheme. The {\it Weil restriction of $X^{\prime}$ along $f$} is the contravariant functor
$(\mathrm{Sch}/S)\to(\mathrm{Sets}), T\mapsto\hom_{
S^{\le\prime}}(T\times_{S}S^{\le\prime},X^{\le\prime}\le)$. This functor is {\it representable} if there exist an $S$-scheme $R_{S^{\le\prime}\be/S}(X^{\prime}\e)$ and a morphism of $S^{\e\prime}$-schemes $\theta_{X^{\prime},\e S^{\le\prime}\be/S}\colon R_{S^{\le\prime}\be/S}(X^{\prime}\e)_{S^{\le\prime}}\to X^{\prime}$ such that the map
\begin{equation}\label{wr}
\hom_{\le S}\e(T,R_{S^{\le\prime}\be/S}(X^{\le\prime}\e))\to\hom_{
S^{\le\prime}}(\e T\!\times_{S}\!S^{\e\prime},X^{\le\prime}\e), \quad g\mapsto \theta_{ X^{\prime}\!,\, S^{\le\prime}\be/S}\circ g_{\le S^{\le\prime}},
\end{equation}
is a bijection (functorially in $T\e$). See \cite[\S7.6]{blr} for basic information on the Weil restriction functor. We will write
\begin{equation}\label{imor}
j_{ X,\e S^{\lle\prime}\be/ S}\colon X\to R_{\e S^{\lle\prime}\be/ S}(X_{\be S^{\lle\prime}}\be)
\end{equation}
for the canonical adjunction $S$-morphism, i.e., the $S$-morphism that corresponds to the identity morphism of $X_{S^{\lle\prime}}$ under the bijection \eqref{wr}.

\section{The projective and norm one groups}\label{pn1}

Let $f\colon \spp\to S$ be a finite and faithfully flat morphism of locally noetherian schemes of constant rank $n\geq 2$ and let $G$ be a smooth, commutative and quasi-projective $S$-group scheme with connected fibers. The map $j_{ G,\e S^{\lle\prime}\be/ S}\colon G\to R_{\e S^{\lle\prime}\be/ S}(G_{\lbe S^{\lle\prime}}\be)$ \eqref{imor} is a closed immersion of smooth, commutative and quasi-projective $S$-group schemes with connected fibers. See \cite[Lemma 3.1]{ga18b} and \cite[\S6.4, Theorem 1, p.~153; \S7.6, Theorem 4 and Proposition 5, (b) and (h), pp.~194-195]{blr}. The quotient fppf sheaf of abelian groups on $S$
\begin{equation}\label{pg}
P_{\lbe S^{\lle\prime}\be/\lbe S}(G\e)=\cok[\e G\overset{j_{\le G,\lle S^{\lle\prime}\!\lbe/\lbe S}}{\lra}R_{\e S^{\lle\prime}\be/ S}(G_{\lbe S^{\lle\prime}}\be)\,]
\end{equation}
is called the {\it projective group associated to $(\e f,G\e)$} (after Voskresenskii \cite[p.~198, line 6]{v}). If ${\rm dim}\e S\leq 1$, \eqref{pg} is represented by a smooth, commutative and quasi-projective $S$-group scheme with connected fibers. See \cite[Theorem 4.C, p.~53]{an}, \cite[\S6.4, Theorem 1, p.~153]{blr}, \cite[${\rm VI_{B}}$, Proposition 9.2, (x) and (xii)]{sga3} and \cite[Lemma 2.54]{gfr}. 

\smallskip

Next let 
\begin{equation}\label{tmor}
N_{G,\e S^{\lle\prime}\be/ S}\colon R_{\e S^{\lle\prime}\be/ S}(G_{\lbe S^{\lle\prime}}\be)\to G
\end{equation}
be the norm morphism  defined in \cite[XVII, 6.3.13.1 and 6.3.14(a)]{sga4}. By \cite[Proposition 3.2]{ga18b}, \eqref{tmor} is a smooth surjection. Further, by \cite[XVII, Proposition 6.3.15(iv)]{sga4}, the $n$-th power morphism on $G$ factors as
\begin{equation}\label{nis}
n_{\le G}\colon G\overset{j_{\e G,\le S^{\lle\prime}\be/ S}}{\hookrightarrow } R_{\e S^{\lle\prime}\be/ S}(G_{\lbe S^{\lle\prime}}\be)\overset{\!N_{G,\e S^{\lle\prime}\be/S}}{\lra} G.
\end{equation}
The {\it norm one group scheme associated to $(\e f,G\e)$} is the $S$-group scheme
\begin{equation}\label{n1}
R_{\e S^{\lle\prime}\!/\lbe S}^{\e(1)}\lbe(G\e)=\krn[R_{\e S^{\lle\prime}\be/ S}(G_{\lbe S^{\lle\prime}}\be)\overset{\!\be N_{\lbe G,\le S^{\lle\prime}\!\lbe/ S}}{\lra} G\,].
\end{equation}
By \cite[Proposition 3.10]{ga18b}, $R_{\e S^{\lle\prime}\!/\lbe S}^{\e(1)}\lbe(G\e)$ is smooth and commutative with connected fibers. Further, since $G$ is quasi-projective and therefore separated over $S$ \cite[comment after Definition 5.3.1]{ega2}, the unit section $\varepsilon\colon S\to G$ is a closed immersion \cite[${\rm VI_{B}}$, Proposition 5.1]{sga3}. Consequently, the canonical $S$-morphism $R_{\e S^{\lle\prime}\!/\lbe S}^{\e(1)}\lbe(G\e)\to R_{\e S^{\lle\prime}\!/\lbe S}\lbe(G\e)$ is a closed immersion whence, by 
\cite[Propositions 5.3.1, (i) and (ii), p.~279, and 6.3.4, (i) and (ii), p.~304]{ega1}, $R_{\e S^{\lle\prime}\!/\lbe S}^{\e(1)}(G\e)$ is separated and of finite type over $S$. Now \cite[\S6.4, Theorem 1, p.~153]{blr} shows that $R_{\e S^{\lle\prime}\!/\lbe S}^{\e(1)}\lbe(G\e)$ is quasi-projective over $S$.

\begin{remark} When $S=\spec A$, where $A$ is either a global field or a ring of integers in such a field, and $G=\bg_{m,S}$, the groups \eqref{pg} and \eqref{n1} have been discussed by Shyr \cite[\S5]{sh77}, \cite{sh79}, Ono \cite{o85, o87}, Katayama \cite{k86a,k86b,k87,k89,k91}, Sasaki \cite{sa88}, Morishita \cite{m91}, Voskresenskii \cite[Chapter 7, \S20]{v} and Liu and Lorenzini \cite[\S4]{ll01}. 
\end{remark}

\smallskip

There exist canonical exact sequences in $\sfs$
\begin{equation}\label{se1}
0\to G\overset{j_{G,S^{\prime}\!\lbe/\lbe S}}{\lra}R_{\e S^{\lle\prime}\be/ S}(G_{\lbe S^{\lle\prime}}\be)\overset{\! q}{\lra}P_{\lbe S^{\lle\prime}\be/\lbe S}(G\e)\to 0
\end{equation}
and
\begin{equation}\label{se2}
0\to R_{\e S^{\lle\prime}\!/\lbe S}^{\e(1)}\lbe(G\e)\overset{\!a}{\lra}R_{\e S^{\lle\prime}\be/ S}(G_{\lbe S^{\lle\prime}}\lbe)\overset{\!\be N_{\lbe\lbe G,\le S^{\lle\prime}\!\lbe/ S}}{\lra} G\to 0,
\end{equation}
where $q$ is the canonical projection and $a$ is the inclusion. Now, if $F=G, R_{\e S^{\lle\prime}\be/ S}(G_{\lbe S^{\lle\prime}}\be)$ or $R_{\e S^{\lle\prime}\!/\lbe S}^{\e(1)}\lbe(G\e)$, then the canonical map $H^{r}\lbe(S_{\et},F\e)\to H^{r}\lbe(S_{\fl},F\e)$ is an isomorphism of abelian groups \cite[Theorem 11.7(1), p.~180]{dix}, whence \eqref{se1} and \eqref{se2} induce exact sequences of abelian groups
\begin{equation}\label{lse1}
\begin{array}{rcl}
\dots&\overset{\partial^{\le(\lle r-1\lle)}}{\lra}& H^{\le r}\be(S_{\et},G\e)\overset{j^{(\lle r\lle)}}{\lra} H^{\le r}\be(S_{\et},R_{\e S^{\lle\prime}\be/ S}(G_{\lbe S^{\lle\prime}}\be)\e)\overset{q^{(\lle r\lle)}}{\lra}H^{\le r}\be(S_{\fl},P_{\lbe S^{\lle\prime}\be/ S}(G_{\lbe S^{\lle\prime}}\be)\e)\\
&\overset{\partial^{(\lle r\lle)}}{\lra}& H^{\le r+1}\be(S_{\et},G\e)\overset{j^{(\lle r+1\lle)}}{\lra}\dots
\end{array}
\end{equation}
and
\begin{equation}\label{lse2}
\begin{array}{rcl}
\dots&\overset{\delta^{\le(\lle r-1\lle)}}{\lra}& H^{\le r}\be(S_{\et},R_{\e S^{\lle\prime}\!/\lbe S}^{\e(1)}\lbe(G\e))\overset{a^{(\lle r\lle)}}{\lra} H^{\le r}\be(S_{\et},R_{\e S^{\lle\prime}\be/ S}(G_{\lbe S^{\lle\prime}}\be)\e)\overset{N^{(\lle r\lle)}}{\lra}H^{\le r}\be(S_{\et},G\e)\\
&\overset{\delta^{(\lle r\lle)}}{\lra}& H^{\le r+1}\be(S_{\et},R_{\e S^{\lle\prime}\!/\lbe S}^{\e(1)}\lbe(G\e))\overset{a^{(\lle r+1\lle)}}{\lra}\dots,
\end{array}
\end{equation}
where $j^{(\lle r\lle)}=H^{\le r}\be(S_{\et},j_{G,S^{\prime}\!\lbe/\lbe S}), a^{(\lle r\lle)}=H^{\le r}\be(S_{\et},a)$, $N^{(r)}=H^{\le r}\be(S_{\et},N_{G,S^{\lle\prime}\be/ S})$ and $q^{(\lle r\lle)}$ is the composition
\[
H^{\le r}\be(S_{\et},R_{\e S^{\lle\prime}\be/ S}(G_{\lbe S^{\lle\prime}}\be)\e)\isoto H^{\le r}\be(S_{\fl},R_{\e S^{\lle\prime}\be/ S}(G_{\lbe S^{\lle\prime}}\be)\e)\overset{\!H^{r}\be(S_{\fl},q)}{\lra}H^{\le r}\be(S_{\fl},P_{\lbe S^{\lle\prime}\be/ S}(G_{\lbe S^{\lle\prime}}\be)\e).
\]

Now let $e_{(r)}\colon H^{\e r}\lbe(S_{\et},R_{\e S^{\lle\prime}\be/ S}(G_{\lbe S^{\lle\prime}}\be)\e)\isoto H^{\e r}\lbe(\spp_{\et},G\e)$ be the isomorphism in \cite[Theorem 6.4.2(ii), p.~128]{t} and consider the compositions
\begin{equation}\label{res}
{\rm Res}_{\e G}^{(r)}\colon H^{\le r}\be(S_{\et},G\e)\overset{\! j^{(r)}}{\lra} H^{\e r}\lbe(S_{\et},R_{\e S^{\lle\prime}\be/ S}(G_{\lbe S^{\lle\prime}}\be)\e)\overset{e_{(r)}}{\underset{\sim}{\lra}} H^{\e r}\lbe(\spp_{\et},G\e)
\end{equation}
and
\begin{equation}\label{cores}
{\rm Cores}_{\e G}^{(r)}\colon H^{\e r}\be(\spp_{\et},G\e)\overset{e_{(r)}^{-1}}{\underset{\sim}{\lra}}H^{\e r}\lbe(S_{\et},R_{\e S^{\lle\prime}\be/ S}(G_{\lbe S^{\lle\prime}}\be)\e)\overset{\!N^{(\e r\le)}}{\lra}H^{\le r}\lbe(S_{\et},G\e).
\end{equation}
Occasionally, we will write
\begin{eqnarray}
\img\e {\rm Res}_{\e G}^{(r)}&=&{\rm Res}_{\le\spp\!/\lbe S}H^{\le r}\be(S_{\et},G\e)\label{eas1}\\
\krn\e{\rm Cores}_{\e G}^{(r)}&=&{}_{N}H^{\e r}\be(\spp_{\et},G\e)\label{eas2}\\
\img\e {\rm Cores}_{\e G}^{(r)}&=&N_{\spp\!\lbe/\lbe S}H^{\e r}\be(\spp_{\et},G\e)\label{eas3}.
\end{eqnarray}
Note that 
\begin{equation}\label{res0}
\krn\e j^{(r)}=\krn\e {\rm Res}_{\e G}^{(r)}.
\end{equation}
Further, $e_{(r)}$ induces isomorphisms of abelian groups
\begin{equation}\label{mois}
\krn\e N^{(r)}\isoto\krn\e {\rm Cores}_{\e G}^{(r)}
\end{equation}
and
\begin{equation}\label{cores0}
\cok\e N^{(\e r\le)}=\cok\e {\rm Cores}_{\e G}^{(r)}.
\end{equation}

\begin{remarks}\label{ntor}\indent
\begin{enumerate}
\item[(a)] It follows from \eqref{nis} that the composition ${\rm Cores}_{\e G}^{(r)}\circ {\rm Res}_{\e G}^{(r)}$ is the multiplication by $n$ morphism on $H^{\le r}\be(S_{\et},G\e)$. Consequently, $\cok\e {\rm Cores}_{\e G}^{(r)}$ and
\[
\krn\e {\rm Res}_{\e G}^{(r)}=\krn[H^{\le r}\be(S_{\et},G\e)_{n}\to H^{\le r}\be(\spp_{\et},G\e)]
\]
are $n$-torsion abelian groups.
\item[(b)] If $n$ is invertible on $S$ and the \'etale $\ell$-cohomological dimension of $S$, ${\rm cd}_{\e\ell}(S_{\et})$, exists for every prime $\ell$ that divides $n$, then 
$\krn\e {\rm Res}_{\e G}^{(r)}=0$ for every $r>\max_{\,\ell\e|\lle n}\be{\rm cd}_{\e\ell}(S_{\et})$. Indeed, by (a), $\krn\e{\rm Res}_{\e G}^{(r)}$  is the kernel of the map $H^{\lle r}\be(S_{\et},G\e)_{n}\to H^{\lle r}\lbe(\spp_{\et},G\e)_{n}$ induced by ${\rm Res}_{\e G}^{(r)}$. Now the proof of \cite[Proposition 3.3]{ga18a} shows that $n\colon G\to G$ is \'etale, surjective and locally of finite presentation, whence
the sequence $0\to G_{\lbe n}\to G\overset{\!n}{\to}G\to 0$ is exact in $\ses$ \cite[IV, Proposition 4.4.3]{sga3}. The cohomology sequence associated to the preceding exact sequence shows that $H^{\le r}\be(S_{\et},G\e)_{n}$ is a quotient of $H^{\le r}\be(S_{\et},G_{\lbe n})$, which vanishes if $r>\max_{\,\ell\le| n}\be{\rm cd}_{\e\ell}(S_{\et})$.
\end{enumerate}
\end{remarks}

\smallskip

Next, since $e_{(r)}^{-1}(\e \img\e {\rm Res}_{\e G}^{(r)})=\img\e j^{(r)}=\krn\e q^{(r)}$ by the exactness of \eqref{lse1}, the composition
\[
H^{\e r}\lbe(\spp_{\et},G\e)\overset{e_{(r)}^{-1}}{\underset{\sim}{\lra}}H^{\e r}\lbe(S_{\et},R_{\e S^{\lle\prime}\be/ S}(G_{\lbe S^{\lle\prime}}\be)\e)\overset{q^{(r)}}{\lra}  H^{\e r}\lbe(S_{\fl},P_{\lbe S^{\lle\prime}\be/\lbe S}(G\e)\le)
\]
induces an injection
\begin{equation}\label{qer}
\overline{q}^{\e(r)}\colon \cok\e {\rm Res}_{\e G}^{(r)}\hookrightarrow  H^{\e r}\be(S_{\fl},P_{\lbe S^{\lle\prime}\be/\lbe S}(G\e)\le)
\end{equation}
such that
\begin{equation}\label{qbar0}
\img\e \overline{q}^{\le(r)}=\img\e q^{\le(r)}.
\end{equation}
On the other hand, since $\img \partial^{(r)}=\krn j^{(r+1)}=\krn {\rm Res}_{\e G}^{(r+1)}$ by the exactness of \eqref{lse1} and \eqref{res0}, the map $\partial^{\le(\lle r\lle)}$ \eqref{lse1} induces a surjection
\begin{equation}\label{dbar}
\overline{\partial}^{\e(r)}\colon H^{\e r}\be(S_{\fl},P_{\lbe S^{\lle\prime}\be/\lbe S}(G\e)\le)\twoheadrightarrow \krn {\rm Res}_{\e G}^{(r+1)}
\end{equation}
such that 
\begin{equation}\label{par0}
\krn\e \overline{\partial}^{\e(r)}=\krn\e \partial^{\e(r)}.
\end{equation}
Thus, since $\krn\e \partial^{\e(r)}=\img\e q^{\le(r)}$ (again by the exactness of \eqref{lse1}), \eqref{qbar0} and \eqref{par0} show that the following is an exact sequence of abelian groups for every $r\geq 0$\e:
\begin{equation}\label{fund1}
0\to \cok\e {\rm Res}_{\e G}^{(r)}\overset{\overline{q}^{\le(r)}}{\lra}  H^{\le r}\be(S_{\fl},P_{\lbe S^{\lle\prime}\be/\lbe S}(G\e)\le)\overset{\overline{\partial}^{\le(r)}}{\lra} \krn {\rm Res}_{\e G}^{(r+1)}\to 0.
\end{equation}
We will write 
\begin{equation}\label{qrp}
q_{\lle r}^{\e\prime}\colon \krn \overline{\partial}^{\le(r)}\isoto \cok\e {\rm Res}_{\e G}^{(r)}
\end{equation}
for the inverse of $\overline{q}^{\le(r)}\colon \cok\e {\rm Res}_{\e G}^{(r)}\isoto \krn \overline{\partial}^{\le(r)}$. Thus
\begin{equation}\label{evi}
\overline{q}^{\le(r)}\circ q_{\lle r}^{\e\prime}=1_{\krn \overline{\partial}^{\le(r)}}\quad\text{and}\quad q_{\lle r}^{\e\prime}\circ \overline{q}^{\le(r)}=1_{\cok {\rm Res}_{\e G}^{(r)}}.
\end{equation}

Next, assume that $r\geq 1$. Since $\img\e a^{(r)}=\krn\e N^{(r)}$ by the exactness of \eqref{lse2}, \eqref{mois} shows that the composition
\[
H^{\le r}\be(S_{\et},R_{\e S^{\lle\prime}\!/\lbe S}^{\e(1)}\lbe(G\e))\overset{a^{(r)}}{\lra}H^{\e r}\lbe(S_{\et},R_{\e S^{\lle\prime}\be/ S}(G_{\lbe S^{\lle\prime}}\be)\e)\underset{\!\sim}{\overset{\!e_{(r)}}{\lra}} H^{\e r}\lbe(\spp_{\et},G\e)
\]
induces a surjection
\[
\overline{a}^{\le(r)}\colon H^{\le r}\be(S_{\et},R_{\e S^{\lle\prime}\!/\lbe S}^{\e(1)}\lbe(G\e))\twoheadrightarrow \krn\e {\rm Cores}_{\e G}^{(r)}
\]
such that
\begin{equation}\label{kera}
\krn\e \overline{a}^{\le(r)}=\krn a^{(r)}.
\end{equation}
On the other hand, since $\img \e N^{(r-1)}=\krn\e\delta^{(r-1)}$ by the exactness of \eqref{lse1}, \eqref{cores0} shows that the map $\delta^{(r-1)}$ induces an injection
\[
\overline{\delta}^{\e(r-1)}\colon \cok\e {\rm Cores}_{\e G}^{(r-1)}\hookrightarrow H^{\le r}\be(S_{\et},R_{\e S^{\lle\prime}\!/\lbe S}^{\e(1)}\lbe(G\e))
\]
such that
\begin{equation}\label{opar}
\img\e\overline{\delta}^{(r-1)}=\img\e \delta^{(r-1)}.
\end{equation}
Thus, since $\img\e \delta^{(r-1)}=\krn\e a^{(r)}$ (again by the exactness of \eqref{lse2}), \eqref{kera} and \eqref{opar} show that the following is an exact sequence of abelian groups for every $r\geq 1$\e:
\begin{equation}\label{fund2}
0\to\cok\e {\rm Cores}_{\e G}^{(r-1)}\overset{\overline{\delta}^{(r-1)}}{\lra}  H^{\le r}\be(S_{\et},R_{\e S^{\lle\prime}\!/\lbe S}^{\e(1)}\lbe(G\e))\overset{\overline{a}^{\le(r)}}{\lra} \krn\e {\rm Cores}_{\e G}^{(r)}\to 0.
\end{equation}
The map $\overline{a}^{\le(r)}$ induces an isomorphism
\[
\widetilde{a}^{\le(r)}\colon \cok\e \overline{\delta}^{(r-1)}\isoto \krn\e {\rm Cores}_{\e G}^{(r)}
\]
and we will write
\begin{equation}\label{arp}
a_{\lle r}^{\le\prime}\colon \krn\e {\rm Cores}_{\e G}^{(r)}\isoto \cok\e \overline{\delta}^{(r-1)}
\end{equation}
for its inverse. Thus
\begin{equation}\label{evi2}
\widetilde{a}^{\le(r)}\circ a_{\lle r}^{\le\prime}=1_{\krn {\rm Cores}_{\e G}^{(r)}}\quad\text{and}\quad a_{\lle r}^{\le\prime}\circ \widetilde{a}^{\le(r)} =1_{ \cok \overline{\delta}^{(r-1)}}.
\end{equation}

\section{The comparison morphisms}\label{compm}

We keep the hypotheses of the previous section. Thus $f\colon \spp\to S$ is a finite and faithfully flat morphism of locally noetherian schemes of constant rank $n\geq 2$ and $G$ is a smooth, commutative and quasi-projective $S$-group scheme with connected fibers.

\medskip

Let
\begin{equation}\label{phi}
\varphi\colon R_{\e S^{\lle\prime}\!/\lbe S}^{\e(1)}\lbe(G\e)\to P_{\lbe S^{\lle\prime}\be/\lbe S}(G\e) 
\end{equation}
be defined by the commutativity of the diagram
\begin{equation}\label{t2}
\xymatrix{R_{\e S^{\lle\prime}\!/\lbe S}^{\e(1)}\lbe(G\e)\ar[dr]_(.45){a}\ar[rr]^(.5){\varphi}&& P_{\lbe S^{\lle\prime}\be/\lbe S}(G\e)\\
& R_{\le S^{\prime}\be/S}(G_{\lbe S^{\prime}}\be),\ar[ur]_(.5){q}&
}
\end{equation}
where $a$ and $q$ are the maps in \eqref{se1} and \eqref{se2}, respectively, i.e.,
\begin{equation}\label{obv}
\varphi=q\circ a.
\end{equation}

\smallskip

Next, to simplify the notation, set $j=j_{ G,\e S^{\lle\prime}\be/ S}$ and $N=N_{G,\e S^{\lle\prime}\be/ S}$. If $T\to S$ is any object of $S_{\fl}$ and $x\in R_{\e S^{\lle\prime}\be/ S}(G_{\lbe S^{\lle\prime}}\lbe)(\lle T\lle)$, then
\[
x^{n}\lbe(\e j(\lle T\lle)(N\lbe(\lle T\lle)(x)))^{-1}\in \krn\e N\lbe(\lle T\lle)=\img\e a(\lle T\lle),
\]
by the factorization \eqref{nis} and the exactness of \eqref{se2}. Thus we obtain a morphism in $\sfs$
\begin{equation}\label{bmor}
b\colon R_{\e S^{\lle\prime}\be/ S}(G_{\lbe S^{\lle\prime}}\lbe)\to R_{\e S^{\lle\prime}\!/\lbe S}^{\e(1)}\lbe(G\e),\, x\mapsto x^{n}(\e j(N\lbe(x)))^{-1},
\end{equation}
such that
\begin{eqnarray}
b\circ j&=&0\label{bj}\\
b\circ a&=&n_{\lbe R^{\le(1)}}\label{pta}\\
q\circ a\circ b&=&n_{\lbe P}\be\circ\lbe q, \label{forx}
\end{eqnarray}
where $n_{\lbe R^{\le(1)}}$ and $n_{\lbe P}$ denote the $n$-th power morphisms on $R_{\e S^{\lle\prime}\!/\lbe S}^{\e(1)}\lbe(G\e)$ and $P_{\e S^{\lle\prime}\!/\lbe S}\lbe(G\e)$, respectively. We will write
\begin{equation}\label{bmor0}
b_{\lle 0}=b(\lbe S\le)\colon G(\spp\le)\to {}_{N}G(\spp\le),\, x\mapsto x^{n}(\e N_{\lbe S^{\lle\prime}\be/ S}(x))^{-1},
\end{equation}
where ${}_{N}G(\spp\le)=\krn\e{\rm Cores}_{G}^{(0)}$ \eqref{eas2}.

\begin{remark} When $S$ is the spectrum of a global field and $G$ is a particular type of  $S$-torus, the map $b$ \eqref{bmor} was considered by Shyr \cite[\S5, (15)]{sh77}, Ono \cite{o85} and Katayama \cite{k86b}.
\end{remark}

By \eqref{bj} and the  exactness of \eqref{se2}, there exists a morphism in $\sfs$
\begin{equation}\label{psi}
\psi\colon P_{\lbe S^{\lle\prime}\be/\lbe S}(G\e)\to R_{\e S^{\lle\prime}\!/\lbe S}^{\e(1)}\lbe(G\e)
\end{equation}
such that the following diagram commutes
\begin{equation}\label{t1}
\xymatrix{R_{\le S^{\prime}\be/S}(G_{\lbe S^{\prime}})\ar[dr]_(.45){q}\ar[rr]^(.5){b}&& R_{\e S^{\lle\prime}\!/\lbe S}^{\e(1)}\lbe(G\e).\\
& P_{\lbe S^{\lle\prime}\be/\lbe S}(G\e)\ar[ur]_(.5){\psi}&
}
\end{equation}
By \eqref{obv}, \eqref{pta} and \eqref{t1}, we have 
\begin{equation}\label{d2}
\psi\circ\varphi=n_{\lbe R^{\le(1)}}.
\end{equation}
Further, since $q$ is an epimorphism in $S_{\fl}^{\le\sim}$ and 
$\varphi\circ\psi\circ q= q\circ a\circ b=n_{P}\circ q$ by \eqref{forx}, we have
\begin{equation}\label{d1}
\hskip -.5cm\varphi\circ\psi=n_{\lbe P}.
\end{equation}

The maps \eqref{phi} and \eqref{psi} are called the {\it comparison morphisms} associated to the pair $(\e f,G\e)$.

\smallskip

\begin{lemma}\label{fi} There exist canonical isomorphisms in $S_{\fl}^{\le\sim}$
\[
\krn\varphi\simeq G_{\lbe n}\qquad{\text{and}} \qquad\cok\varphi\simeq G/n,
\]
where $\varphi$ is the map \eqref{phi}.
\end{lemma}
\begin{proof} An application of Lemma \ref{ker-cok} to the pair of morphisms
\[
R_{\e S^{\lle\prime}\!/\lbe S}^{\e(1)}\lbe(G\e)\overset{\!a}{\hookrightarrow}
R_{\le S^{\prime}\be/S}(G_{\lbe S^{\prime}}\be)\overset{\!q}{\twoheadrightarrow}P_{\lbe S^{\lle\prime}\be/\lbe S}(G\e),
\]
whose composition equals $\varphi$ by the commutativity of \eqref{t2}, yields the top row of the following exact and commutative diagram in $S_{\fl}^{\le\sim}$
\[
\xymatrix{0\ar[r]&\krn\varphi\ar[r]^{a^{\prime}}&\krn q\ar[r]^(.45){\kappa}&\cok\e a\ar[d]^{N^{\le\prime}}_{\simeq}\ar[r]^{q^{\le\prime}}&\cok\varphi\ar[r]& 0\\
&&G\ar[u]^{j^{\le\prime}}_{\simeq}\ar[r]^{n}&G.&&
}
\]
The map $a^{\lle\prime}$ is induced by $a$, $\kappa$ is the composition $\krn q\hookrightarrow R_{\le S^{\prime}\be/S}(G_{\lbe S^{\prime}}\be)\twoheadrightarrow \cok\e a$, $q^{\le\prime}$ is induced by $q$, the  vertical isomorphisms are induced by $j_{ G,\e S^{\lle\prime}\be/ S}$ and $N_{ G,\e S^{\lle\prime}\be/ S}$ and the square commutes by \eqref{nis}. The maps $a^{\lle\prime}$ and $j^{\le\prime}$ induce isomorphisms $a^{\lle\prime\prime}\colon\krn\varphi\isoto \krn\e\kappa$ and $j^{\le\prime\prime}\colon G_{\lbe n}\isoto \krn\e\kappa$ such that the following diagram commutes
\begin{equation}\label{top}
\xymatrix{\ar@{^{(}->}[d]G_{\lbe n}\ar[r]^(.4){j^{\le\prime\prime}}_(.4){\sim}&
\krn\e\kappa\ar@{^{(}->}[d]\ar[r]_(.45){\sim}^{(a^{\lle\prime\prime})^{-1}}&\krn\varphi\,
\ar@{^{(}->}[r]&R_{\e S^{\lle\prime}\!/\lbe S}^{\e(1)}\lbe(G\e)_{n}\ar@{^{(}->}[d]\\
G\ar[r]^(.4){j^{\le\prime}}_(.4){\sim}&\krn q\,\ar@{^{(}->}[r]&R_{\le S^{\prime}\be/S}(G_{\lbe S^{\prime}}\be)&R_{\e S^{\lle\prime}\!/\lbe S}^{\e(1)}\lbe(G\e).\ar@{_{(}->}[l]_(.45){a}
}
\end{equation}
Further, the map $N^{\le\prime}$ induces an isomorphism $N^{\le\prime\prime}\colon\cok\e\kappa\isoto G/n$. The isomorphisms of the lemma are defined by the commutativity of the following triangles:
\[
\xymatrix{G_{\lbe n}\ar[dr]^(.5){\sim}\ar[rr]^(.45){j^{\le\prime\prime}}_(.45){\sim}&& \krn\e\kappa\\
&\krn\e\varphi\ar[ur]_(.55){a^{\le\prime\prime}}^(.45){\sim}
}\qquad\xymatrix{\cok\e\kappa\ar[dr]_(.4){N^{\le\prime\prime}}^(.5){\sim}\ar[rr]^{q^{\le\prime\prime}}_(.48){\sim}&&\cok\e\varphi\\
&G/n\ar[ur]^{\sim}},
\]
where the map $q^{\le\prime\prime}$ is induced by $q^{\le\prime}$.
\end{proof}

Let
\begin{equation}\label{gn}
G(n)=\cok[\e G_{\lbe n}\hookrightarrow R_{\e S^{\lle\prime}\!/\lbe S}^{\e(1)}\lbe(G\e)_{n}\e]
\end{equation}
be the cokernel (in $S_{\fl}^{\le\sim}$) of the composition of the top horizontal arrows in diagram \eqref{top}. Further, let
\begin{equation}\label{inc}
\iota\colon G_{\lbe n}\hookrightarrow R_{\e S^{\lle\prime}\!/\lbe S}^{\e(1)}\lbe(G\e)
\end{equation}
be the composition $G_{\lbe n} \hookrightarrow R_{\e S^{\lle\prime}\!/\lbe S}^{\e(1)}\lbe(G\le)_{\lbe n}\hookrightarrow R_{\e S^{\lle\prime}\!/\lbe S}^{\e(1)}\lbe(G\le)$. The sheaf \eqref{gn} was introduced in \cite[p.~15]{ga18b}

\begin{lemma}\label{pss} There exists a canonical exact sequence in $S_{\fl}^{\le\sim}$
\[
0\to G(n)\to\krn \psi\to G/n\to R_{\e S^{\lle\prime}\!/\lbe S}^{\e(1)}\lbe(G\e)/n\to\cok\psi\to 0,
\]
where $\psi$ is the comparison morphism \eqref{psi}.
\end{lemma}
\begin{proof} We apply Lemma \ref{ker-cok} to the pair of morphisms
\[
R_{\e S^{\lle\prime}\!/\lbe S}^{\e(1)}\lbe(G\e)\overset{\!\varphi}{\to}P_{\lbe S^{\lle\prime}\be/\lbe S}(G\e)\overset{\!\psi}{\to}R_{\e S^{\lle\prime}\!/\lbe S}^{\e(1)}\lbe(G\e),
\]
whose composition is $n_{\lbe R^{\le(1)}}$ \eqref{d2}. We then use Lemma \eqref{fi} and the definition \eqref{gn} to obtain the sequence of the lemma.
\end{proof}

\smallskip

We now recall from \cite{ga18b} the following

\begin{definition}\label{adm} The pair $(\e f,G\e)$ is called {\it admissible} if \begin{enumerate}
\item[(i)] $f\colon \spp\to S$ is a finite and faithfully flat morphism of locally noetherian schemes of constant rank $n\geq 2$,
\item[(ii)] $G$ is a smooth, commutative and quasi-projective $S$-group scheme with connected fibers, and
\item[(iii)] for every point $s\in S$ such that ${\rm char}\, k(\lbe s\lbe)$ divides $n$,
\begin{enumerate}
\item[(iii.1)] $G_{k(\lbe s\lbe)}$ is a semiabelian $k(\lbe s\lbe)$-variety, and
\item[(iii.2)] $f_{\lbe s}\colon \spp\times_{S}\spec k(\lbe s\lbe)\to \spec k(\lbe s\lbe)$ is \'etale.
\end{enumerate}
\end{enumerate}
\end{definition}

\smallskip

Note that, if $n$ is invertible on $S$, then condition (iii) above is vacuous. Examples of admissible pairs can be obtained by adding condition (iii.2) above to the examples in \cite[Examples 3.2]{ga18a}. Admissible pairs were introduced in \cite{ga18b} so that the following statement holds.

\smallskip

\begin{proposition}\label{isg2} Assume that $(\e f, G\e)=(\spp\!/S, G\e)$ is admissible (see Definition {\rm \ref{adm}}) and let $n\geq 2$ be the rank of $f$.
If $H=G,\e R_{\e S^{\lle\prime}\be/ S}(G_{\lbe S^{\lle\prime}}\be)$ or $R_{\e S^{\lle\prime}\!/\lbe S}^{\e(1)}\lbe(G\e)$, then $n\colon H\to H$ is an epimorphism in $\sfs$.
\end{proposition}
\begin{proof} See \cite[Proposition 3.12]{ga18b} and \cite[IV, Proposition 4.4.3]{sga3}.
\end{proof}

The next statement is immediate from Proposition \ref{isg2} and Lemmas \ref{fi} and \ref{pss}.

\begin{proposition} Assume that the pair $(\e f, G\e)$ is admissible (see Definition {\rm \ref{adm}}). Then there exist canonical exact sequences in $\sfs$
\begin{equation}\label{sap1}
0\to G_{\lbe n}\overset{\!\iota}{\to} R_{\e S^{\lle\prime}\!/\lbe S}^{\e(1)}\lbe(G\e)\overset{\!\varphi}{\to} P_{\lbe S^{\lle\prime}\be/\lbe S}(G\e)\to 0
\end{equation}
and
\begin{equation}\label{sap2}
0\to G(n)\overset{\!\be\ell}{\to} P_{\lbe S^{\lle\prime}\be/\lbe S}(G\e)\overset{\!\psi}{\to}  R_{\e S^{\lle\prime}\!/\lbe S}^{\e(1)}\lbe(G\e)\to 0,
\end{equation}
where $\iota$ is the map \eqref{inc}, $\varphi$ is defined by \eqref{t2}, $G(n)$ is the sheaf \eqref{gn}, $\ell$ is induced by $\varphi_{n}$ and $\psi$ is defined by \eqref{t1}. 
\end{proposition}

If we identify the groups $H^{\le r}\be(S_{\et},R_{\e S^{\lle\prime}\!/\lbe S}^{\e(1)}\lbe(G\e))$ and $H^{\le r}\be(S_{\fl},R_{\e S^{\lle\prime}\!/\lbe S}^{\e(1)}\lbe(G\e))$ for every $r\geq 0$ via \cite[Theorem 11.7(1), p.~180]{dix}, then the proposition shows that, if $(\e f, G\e)$ is admissible, there exist canonical exact sequences of abelian groups
\begin{equation}\label{phis}
\begin{array}{rcl}
\dots&\to& H^{\le r}\be(S_{\fl},G_{\lbe n})\overset{\iota^{(\lle r\lle)}}{\lra} H^{\le r}\be(S_{\et},R_{\e S^{\lle\prime}\!/\lbe S}^{\e(1)}\lbe(G\e))\overset{\varphi^{(\lle r\lle)}}{\lra}H^{\le r}\be(S_{\fl},P_{\lbe S^{\lle\prime}\be/S}(G\e))\\
&\to& H^{\le r+1}\be(S_{\fl},G_{\lbe n})\overset{\iota^{(\lle r+1\lle)}}{\lra} H^{\le r+1}\be(S_{\et},R_{\e S^{\lle\prime}\!/\lbe S}^{\e(1)}\lbe(G\e))\dots
\end{array}
\end{equation}
and
\begin{equation}\label{psis}
\begin{array}{rcl}
\dots&\to& H^{\le r}\be(S_{\fl},G(n))\overset{\ell^{\lle(\lle r\lle)}}{\lra} H^{\le r}\be(S_{\fl},P_{\lbe S^{\lle\prime}\be/S}(G\e))\overset{\psi^{(\lle r\lle)}}{\lra}H^{\le r}\be(S_{\et},R_{\e S^{\lle\prime}\!/\lbe S}^{\e(1)}\lbe(G\e))\\
&\to& H^{\le r+1}\be(S_{\fl},G(n))\overset{\ell^{\lle(\lle r+1\lle)}}{\lra}H^{\le r+1}\be(S_{\fl},P_{\lbe S^{\lle\prime}\be/S}(G\e))\to\dots.
\end{array}
\end{equation}
Clearly, \eqref{phis} and \eqref{psis} yield isomorphisms of abelian groups 
\begin{equation}\label{clear1}
\cok\e\varphi^{(r)}\isoto \krn\e\iota^{(\lle r+1\lle)} 
\end{equation}
and
\begin{equation}\label{clear2}
\cok\e\psi^{(r)}\isoto\krn\e\ell^{\e(\lle r+1\lle)}.
\end{equation}

\section{Restriction cokernels and corestriction kernels}\label{rc}

We continue to assume that $(\e f, G\e)$ is an admissible pair (see Definition \ref{adm}).

By \eqref{opar}, \eqref{obv} and the exactness of \eqref{lse1} and \eqref{lse2}, we have 
\[
\varphi^{(r)}\lbe(\e\img\e \overline{\delta}^{(r-1)})=(\e q^{(r)}\be\circ\be a^{(r)})\lbe(\krn\e a^{(r)})=0
\]
and
\[
\overline{\partial}^{\le(r)}\!\be\circ \varphi^{(r)}=\partial^{\le(r)}\!\be\circ q^{(r)}\circ a^{(r)}=0
\]
for every $r\geq 1$. Thus the following diagram, whose top and bottom rows are, respectively, the exact sequences \eqref{fund2} and \eqref{fund1} commutes:
\begin{equation}\label{bas1}
\xymatrix{0\ar[r]& \ar[d]^(.47){0}\cok\e {\rm Cores}_{\e G}^{(r-1)}\ar[r]^{\overline{\delta}^{(r-1)}}& H^{\le r}\be(S_{\et},R_{\e S^{\lle\prime}\!/\lbe S}^{\e(1)}\lbe(G\e))\ar[d]^(.47){\varphi^{(r)}}\ar[r]^(.57){\overline{a}^{\le(r)}}&\ar[d]^{0}\krn\e {\rm Cores}_{\e G}^{(r)}\ar[r]&0\\
0\ar[r]& \cok\e {\rm Res}_{\e G}^{(r)}\ar[r]^(.45){\overline{q}^{\le(r)}}& H^{\le r}\be(S_{\fl},P_{\lbe S^{\lle\prime}\!/\lbe S}(G\e)\le)\ar[r]^(.57){\overline{\partial}^{\le(r)}}&\krn {\rm Res}_{\e G}^{(r+1)}\ar[r]&0.	
}
\end{equation}
Thus we obtain a canonical morphism of abelian groups
\begin{equation}\label{ar}
\alpha_{\le r}\colon \krn\e {\rm Cores}_{\e G}^{(r)}\to  \cok\e {\rm Res}_{\e G}^{(r)},
\end{equation}
namely the composition
\begin{equation}\label{ard}
\krn\e {\rm Cores}_{\e G}^{(r)}\underset{\sim}{\overset{a_{r}^{\le\prime}}{\lra}}
\cok\e \overline{\delta}^{(r-1)}\overset{\overline{\varphi}^{\le(r)}}{\lra}\krn\e \overline{\partial}^{\le(r)}\underset{\sim}{\overset{q_{r}^{\le\prime}}{\lra}}\cok\e {\rm Res}_{\e G}^{(r)}
\end{equation}
where $a_{\lle r}^{\le\prime}$ and $q_{r}^{\le\prime}$ are the maps \eqref{arp} and \eqref{qrp}, respectively, and $\overline{\varphi}^{\le(r)}$ fits into an exact and commutative diagram
\begin{equation} \label{phit}
\xymatrix{\cok\e {\rm Cores}_{\e G}^{(r-1)}\ar@{^{(}->}[r]^(.5){\overline{\delta}^{(r-1)}}& H^{\le r}\be(S_{\et},R_{\e S^{\lle\prime}\!/\lbe S}^{\e(1)}\lbe(G\e))\ar[d]_{\varphi^{\le(r)}}\ar@{->>}[r]&\cok\e \overline{\delta}^{(r-1)}\ar[dl]^(.45){\overline{\varphi}^{\le(r)}}.\\
&\krn\e \overline{\partial}^{\le(r)}&
}
\end{equation}
The preceding considerations remain valid when $r=0$ if we set
$\overline{\delta}^{\le(-1)}=\cok\e {\rm Cores}_{\e G}^{(-1)}=0$ above. The resulting map
\begin{equation}\label{ar0}
\alpha_{\le 0}\colon {}_{N}G(\lbe\spp\le)\to  G(\spp\le)\be/\lbe G(S\e)
\end{equation}
is the canonical map induced by the inclusion ${}_{N}G(\lbe\spp\le)\hookrightarrow  G(\spp\le)$

\smallskip

Next, let
\begin{equation}\label{br}
\beta_{\le r}\colon \cok\e {\rm Res}_{\e G}^{(r)}\to\krn\e {\rm Cores}_{\e G}^{(r)}
\end{equation}
be the composition of continuous arrows in the following diagram with exact rows
\begin{equation}\label{bas2}
\xymatrix{\cok\e {\rm Res}_{\e G}^{(r)}\ar@{^{(}->}[r]^(.45){\overline{q}^{\le(r)}}& H^{\le r}\be(S_{\fl},P_{\lbe S^{\lle\prime}\be/\lbe S}(G\e)\le)\ar@{..>>}[r]^(.57){\overline{\partial}^{\le(r)}}\ar[d]^(.43){\psi^{\lbe(r)}}&\krn {\rm Res}_{\e G}^{(r+1)}\\
\cok\e {\rm Cores}_{\e G}^{(r-1)}\ar@{^{(}..>}[r]^(.5){\overline{\delta}^{\le(r-1)}}& H^{\le r}\be(S_{\et},R_{\e S^{\lle\prime}\!/\lbe S}^{\e(1)}\lbe(G\e))\ar@{->>}[r]^(.57){\overline{a}^{\le(r)}}&\krn\e {\rm Cores}_{\e G}^{(r)},
}
\end{equation}
where the top and bottom rows are the exact sequences \eqref{fund1} and \eqref{fund2}, respectively. By \eqref{t1}, the map
\begin{equation}\label{br0}
\beta_{\le 0}\colon G(\lbe\spp\le)\be/\lbe G\lbe(\lbe S\le)\to {}_{N}G(\lbe\spp\le)
\end{equation}
is induced by $b_{\lle 0}\colon G(\spp\le)\to {}_{N}G(\spp\le),\, x\mapsto x^{n}(\e N_{\lbe S^{\lle\prime}\be/ S}(x))^{-1}$ \eqref{bmor0}.

The following statement is inmmediate from \eqref{evi}, \eqref{evi2}, \eqref{d2}, \eqref{d1}, \eqref{ar0} and \eqref{br0}.

\begin{theorem} \label{th1} If $(\e f, G\e)$ is an admissible pair (see Definition {\rm \ref{adm}}) and $r\geq 0$ is an integer, then there exist canonical maps $\alpha_{r}\colon \krn\e {\rm Cores}_{\e G}^{(r)}\to \cok\e {\rm Res}_{\e G}^{(r)}$ and $\beta_{\le r}\colon \cok\e {\rm Res}_{\e G}^{(r)}\to \krn\e {\rm Cores}_{\e G}^{(r)}$ such that the compositions
\[
\krn\e {\rm Cores}_{\e G}^{(r)}\overset{\!\alpha_{r}}{\lra}  \cok\e {\rm Res}_{\e G}^{(r)}\overset{\!\beta_{\le r}}{\lra} \krn\e {\rm Cores}_{\e G}^{(r)}
\]
and
\[
\cok\e {\rm Res}_{\e G}^{(r)}\overset{\!\beta_{\le r}}{\lra} \krn\e {\rm Cores}_{\e G}^{(r)}\overset{\!\alpha_{r}}{\lra} \cok\e {\rm Res}_{\e G}^{(r)}
\]
are the multiplication by $n$ maps on $\krn\e {\rm Cores}_{\e G}^{(r)}$ and $\cok\e {\rm Res}_{\e G}^{(r)}$, respectively.
\end{theorem}

\smallskip

The theorem shows that the kernels and cokernels of the maps $\alpha_{r}$ and $\beta_{r}$ are $n$-torsion abelian groups. We will now describe these groups in terms of the flat (\e fppf\e) cohomology groups of the sheaves $G_{\lbe n}$ and $G(n)$ \eqref{gn}.

\smallskip

\begin{proposition}\label{ka} For every $r\geq 0$, there exist a canonical exact sequence of $n$-torsion abelian groups
\[
0\!\to\!\cok\e {\rm Cores}_{\e G}^{(r-1)}\!\to\!\be \img\e \iota^{\lbe(r)}\!\!\to\! \krn\e \alpha_{r}\to 0
\]
and an isomorphism of $n$-torsion abelian groups
\[
\cok\e\alpha_{r}\simeq \krn\e \iota^{\lbe(r+1)},
\]	
where $\iota^{(\lle r\lle)}\colon H^{\le r}\be(S_{\fl},G_{\lbe n}\lbe)\to H^{\le r}\be(S_{\et},R_{\e S^{\lle\prime}\!/\lbe S}^{\e(1)}\lbe(G\e))$ is the map in \eqref{phis}.
\end{proposition}
\begin{proof} By Remark \ref{ntor}(a), $\cok\e {\rm Cores}_{\e G}^{(r-1)}$ is an 
$n$-torsion abelian group. Now, it follows from the definition of $\alpha_{\lle r}$ \eqref{ard} that $a_{r}^{\le\prime}$ \eqref{arp} and $\overline{q}_{r}$ \eqref{qer} induce isomorphisms of abelian groups $\krn\e \alpha_{\lle r}\isoto \krn\e\overline{\varphi}^{(r)}$ and $\cok\e \alpha_{\lle r}\isoto \cok\e\overline{\varphi}^{\le(r)}$. On the other hand, an application of Lemma \ref{ker-cok} to the triangle in diagram \eqref{phit}, using the identity $\krn\e \varphi^{(r)}=\img\e \iota^{\lbe(r)}$ and the isomorphism \eqref{clear1}, yields an isomorphism $\cok\e\overline{\varphi}^{\le(r)}\simeq \krn\e \iota^{\lbe(r+1)}$ and an exact sequence
\[
0\to\cok\e {\rm Cores}_{\e G}^{(r-1)}\to\img\e \iota^{\lbe(r)}\to \krn\e\overline{\varphi}^{\le(r)}\to 0.
\]
The proposition is now clear.
\end{proof}

\smallskip

The analog of Proposition \ref{ka} for the map $\beta_{r}$ is (significantly) more complicated.

Consider the composition
\begin{equation}\label{brp}
\beta_{\le r}^{\lle\prime}\colon H^{\le r}\be(S_{\fl},P_{\lbe S^{\lle\prime}\!/\lbe S}(G\e)\le)\overset{\psi^{\le(r)}}{\lra}H^{\le r}\be(S_{\et},R_{\e S^{\lle\prime}\!/\lbe S}^{\e(1)}\lbe(G\e))\overset{\,\, \overline{a}^{\le(\lbe r\lbe)}}{\twoheadrightarrow}\krn\e {\rm Cores}_{\e G}^{(r)}.
\end{equation}
Then $\beta_{\le r}$ \eqref{br} factors as
\begin{equation}\label{fact1}
\cok\e {\rm Res}_{\e G}^{(r)}\overset{\,\, \overline{q}^{\le(\lbe r\lbe)}}{\hookrightarrow}
H^{\le r}\be(S_{\fl},P_{\lbe S^{\lle\prime}\!/\lbe S}\lbe(G\e))\overset{\!\beta_{\le r}^{\lle\prime}}{\lra}\krn\e {\rm Cores}_{\e G}^{(r)}.
\end{equation}
Now let
\begin{equation}\label{lr}
\lambda^{\lbe(r)}\colon \krn\e\overline{a}^{(r)}\hookrightarrow H^{\le r}\be(S_{\et},R_{\e S^{\lle\prime}\!/\lbe S}^{\e(1)}\lbe(G\e))\twoheadrightarrow\cok\e \psi^{(r)}
\end{equation}
and
\begin{equation}\label{dtil}
\widetilde{\partial}^{\le(r)}\colon \krn\e\beta_{\le r}^{\lle\prime}\hookrightarrow H^{\le r}\be(S_{\fl},P_{\lbe S^{\lle\prime}\!/\lbe S}(G\e)\le)\twoheadrightarrow \cok\e\overline{q}^{\le(r)}\isoto \krn {\rm Res}_{\e G}^{(r+1)},
\end{equation}
be the maps induced by \eqref{brp} and \eqref{fact1}, where the isomorphism in \eqref{dtil} is induced by $\overline{\partial}^{\le(r)}$ \eqref{dbar}. Applying Lemma \ref{ker-cok} and Remark \ref{k1} to the pairs of maps \eqref{brp} and \eqref{fact1}, we obtain canonical exact sequences of abelian groups
\begin{equation}\label{topi}
0\to \krn\e\psi^{(r)}\to \krn\e\beta_{\le r}^{\e\prime}\to\krn\e\overline{a}^{(r)}\overset{\!\lambda^{\lbe(r)}}{\lra}\cok\psi^{\le(r)}\to \cok\e \beta_{\lle r}^{\lle\prime}\to 0
\end{equation}
and
\begin{equation}\label{el1}
0\to \krn\e \beta_{\lle r}\to \krn\e \beta_{\lle r}^{\e\prime}\overset{\!\widetilde{\partial}^{\le(r)}}{\lra}\krn {\rm Res}_{\e G}^{(r+1)}\to \cok\e \beta_{\lle r}\to \cok\e \beta_{\lle r}^{\lle\prime}\to 0,
\end{equation}
where $\lambda^{\lbe(r)}$ and $\widetilde{\partial}^{\le(r)}$ are the compositions \eqref{lr} and \eqref{dtil}, respectively. Clearly, \eqref{el1} induces an isomorphism 
\begin{equation}\label{el2}
\krn\e \beta_{\lle r}\isoto\krn\e \widetilde{\partial}^{\le(r)}
\end{equation}
and a canonical exact sequence of abelian groups
\begin{equation}\label{el3}
0\to \cok\e \widetilde{\partial}^{\le(r)}\to \cok\e \beta_{\lle r}\to \cok\e \beta_{\lle r}^{\lle\prime}\to 0.
\end{equation}

\smallskip

Now let
\begin{equation}\label{dr}
c_{\le r}\colon \cok\e{\rm Cores}_{\e G}^{(r-1)}\to H^{\le r+1}\be(S_{\fl},G(n))
\end{equation}
be the composition
\begin{equation}\label{dr1}
c_{\le r}\colon \cok\e{\rm Cores}_{\e G}^{(r-1)}\overset{c_{ r}^{\lle\prime}}{\lra}\krn\e\ell^{\le(r+1)}\hookrightarrow  H^{\le r+1}\be(S_{\fl},G(n)),
\end{equation}
where the second map is the inclusion and $c_{\le r}^{\e\prime}$ is the composition
\begin{equation}\label{drp}
c_{\le r}^{\e\prime}\colon \cok\e {\rm Cores}_{\e G}^{(r-1)}\underset{\sim}{\overset{\overline{\delta}^{\le(r-1)}}{\lra}} \krn\e\overline{a}^{\le(r)}\overset{\lambda^{\lbe(r)}}{\lra}\cok\e \psi^{\e(r)}\isoto \krn\e\ell^{\le(r+1)},
\end{equation}
where the first isomorphism comes from \eqref{fund2}, $\lambda^{\lbe(r)}$ is the map \eqref{lr} and the second isomorphism is the map \eqref{clear2}. Clearly,  $\overline{\delta}^{\le(r-1)}$ induces an isomorphism of abelian groups $\krn\e c_{\le r}^{\e\prime}=\krn\e c_{\le r}\isoto \krn\e \lambda^{\be(r)}$, whence \eqref{topi} induces an exact sequence of abelian groups 
\begin{equation}\label{topi2}
0\to \krn\e\psi^{(r)}\to \krn\e\beta_{\le r}^{\e\prime}\to\krn\e c_{\le r}\to 0.
\end{equation}

Next let
\begin{equation}\label{cr}
d_{\le r}\colon H^{\le r}\be(S_{\fl},G(n))\to \krn {\rm Res}_{\e G}^{(r+1)}
\end{equation}
be the composition
\begin{equation}\label{new}
H^{\le r}\be(S_{\fl},G(n))\overset{\!d_{\le r}^{\e\prime}}{\lra}\krn\e\beta_{\le r}^{\e\prime}\overset{\!\widetilde{\partial}^{\le(r)}}{\lra}\krn {\rm Res}_{\e G}^{(r+1)},
\end{equation}
where $\widetilde{\partial}^{\le(r)}$ is the map \eqref{dtil}
and $d_{\le r}^{\e\prime}$ is the composition
\begin{equation}\label{crp}
d_{\le r}^{\e\prime}\colon H^{\le r}\be(S_{\fl},G(n))\twoheadrightarrow\img\e\ell^{\le(r)}=\krn\e\psi^{(r)}\hookrightarrow \krn\e\beta_{\lle r}^{\e\prime}\,,
\end{equation}
where the equality holds by the exactness of \eqref{psis} and the injection is that in \eqref{topi}. By the definition of $d_{\le r}^{\e\prime}$ \eqref{crp} and the exactness of \eqref{topi2}, there exists a canonical exact sequence of abelian groups
\begin{equation}\label{nic}
H^{\le r}\be(S_{\fl},G(n))\overset{\!d_{\le r}^{\e\prime}}{\lra}\krn\e\beta_{\lle r}^{\e\prime}\to \krn\e c_{\le r}\to 0.
\end{equation}

\smallskip

Now, by Remark \ref{ntor}(a), the maps \eqref{dr} and \eqref{cr} are morphisms of $n$-torsion abelian groups. Further, there exists a canonical morphism of $n$-torsion abelian groups
\begin{equation}\label{gr}
\gamma_{\lle r}\colon \krn\e c_{\le r}\to\cok\e d_{\lle r}
\end{equation}
such that the following diagram with exact rows commutes: 
\begin{equation}\label{gap}
\xymatrix{H^{\le r}\be(S_{\fl},G(n))\ar@{=}[d]\ar[r]^(.55){d_{\le r}^{\e\prime}}&\krn\e\beta_{\le r}^{\lle\prime}\ar[r]\ar[d]^(.45){\widetilde{\partial}^{\le(r)}}&\ar[d]^(.45){\gamma_{r}}\krn\e c_{\le r}\ar[r]&0\\
H^{\le r}\be(S_{\fl},G(n))\ar[r]^(.52){d_{r}}&\krn {\rm Res}_{\e G}^{(r+1)}\ar[r]&\cok\e d_{\le r}\ar[r]&0,}
\end{equation}
where the top row is the sequence \eqref{nic} and the lef-hand square commutes by the definition of $d_{\le r}$ \eqref{new}.

\begin{proposition}\label{ak} There exists a canonical exact sequence of $n$-torsion abelian groups
\[
\krn\e d_{\lle r}\to \krn\e \beta_{\lle r}\to \krn\e c_{\lle r}\overset{\!\be\gamma_{r}}{\to}\cok\e d_{\lle r}\to \cok\e \beta_{\lle r}\to \cok\e c_{\le r},
\]
where $c_{\le r}$ and $d_{\le r}$ are the maps \eqref{dr} and \eqref{cr}, respectively, and $\gamma_{r}$ is defined by the commutativity of diagram \eqref{gap}.
\end{proposition}
\begin{proof} Applying Lemma \ref{ker-cok} to the pair of maps \eqref{new} and using the isomorphism \eqref{el2} and the exactness of the sequences \eqref{el3} and \eqref{nic}, we obtain a canonical exact sequence of abelian groups
\[
\krn d_{\le r}\to\krn\e\beta_{\le r}\to \krn\e c_{\le r}\overset{\!\be\gamma_{r}}{\to}\cok\e d_{\le r}\to \cok\e\beta_{\le r}\to \cok\e \beta_{\le r}^{\lle\prime}\to 0.
\]
On the other hand, there exists a canonical isomorphism of abelian groups
$\cok\e \beta_{\lle r}^{\e\prime}\isoto\cok\e c_{\le r}^{\e\prime}$ such that the following diagram with exact rows commutes:
\[
\xymatrix{
\krn\e\overline{a}^{(r)}\ar[r]^{\lambda^{(r)}}\ar[d]_(.45){\sim}&\cok\e \psi^{\e(r)}\ar[d]_(.45){\sim}\ar[r]&\cok\e \beta_{\lle r}^{\lle\prime}\ar[d]_(.45){\sim}\ar[r]&0\\
\cok\e{\rm Cores}_{\e G}^{(r-1)}\ar[r]^(.55){c_{\le r}^{\e\prime}}&\krn\e\ell^{\e(r+1)}\ar[r]&\cok\e c_{\le r}^{\e\prime}\ar[r]&0,
}
\]
where the top row is part of the sequence \eqref{topi}, $c_{\le r}^{\e\prime}$ is the composition \eqref{drp}, the left-hand vertical arrow is the inverse of the map $\overline{\delta}^{\le(r-1)}$ and the middle vertical arrow is the map \eqref{clear2}. Finally, the inclusion in \eqref{dr1} induces an injection $\cok\e c_{\lle r}^{\e\prime}\hookrightarrow \cok\e c_{\lle r}$, whence the proposition follows.
\end{proof}

\section{Quadratic Galois coverings}\label{quad}

In order to simplify the statements below, we make the following

\begin{definition}\label{qadm} An {\it admissible quadratic Galois pair} is a pair $(\e f,G\e)$ where
\begin{enumerate}
\item[(i)] $f\colon \spp\to S$ is a quadratic (i.e., of constant rank $2$) Galois covering of locally noetherian schemes (in particular, $f$ is \'etale),
\item[(ii)] $G$ is a smooth, commutative and quasi-projective $S$-group scheme with connected fibers, and
\item[(iii)] for every point $s\in S$ such that ${\rm char}\, k(\lbe s\lbe)=2$, $G_{k(\lbe s\lbe)}$ is a semiabelian $k(\lbe s\lbe)$-variety.
\end{enumerate}
\end{definition}

Clearly, an admissible quadratic Galois pair is admissible in the sense of Definition \ref{adm}. It was shown in \cite[\S5]{ga18b} that, if $f\colon \spp\to S$ is a Galois covering of constant rank $n\geq 2$, then the fppf sheaf $G(n)$ \eqref{gn} is an $(\spp\!/\lbe S\e)$-form of $G_{\lbe n}^{\e n-2}$, i.e., there exists an isomorphism of $\spp$-group schemes $G(n)_{S^{\lle\prime}}\isoto G_{\lbe n,\e S^{\lle\prime}}^{\e n-2}$. Consequently, if $(\e f,G\e)$ is an admissible quadratic Galois pair, then $G(n)=G(2)=0$
and \eqref{dr} and \eqref{cr} yield equalities
$\krn\e d_{\lle r}=\cok\e c_{\lle r}=0$, $\krn\e c_{\lle r}=\cok\e{\rm Cores}_{\e G}^{(r-1)}$ and $\cok\e d_{\lle r}=\krn {\rm Res}_{\e G}^{(r+1)}$ for every $r\geq 0$. Thus \eqref{gr} is a morphism of $2$-torsion abelian groups
\begin{equation}\label{gr2}
\gamma_{r}\colon \cok\e {\rm Cores}_{\e G}^{(r-1)}\to \krn {\rm Res}_{\e G}^{(r+1)}
\end{equation}
which can be described as follows. By the vanishing of $G(2)$ and the exactness of \eqref{sap2}, the map $\psi\colon P_{\lbe S^{\lle\prime}\be/\lbe S}(G\e)\to R_{\e S^{\lle\prime}\!/\lbe S}^{\e(1)}\lbe(G\e)$ \eqref{psi} is an isomorphism in $\sfs$, whence the induced morphism of abelian groups
$\psi^{(\lle r\lle)}\colon H^{\le r}\be(S_{\fl},P_{\lbe S^{\lle\prime}\be/S}(G\e))\isoto H^{\le r}\be(S_{\et},R_{\e S^{\lle\prime}\!/\lbe S}^{\e(1)}\lbe(G\e))$ is an isomorphism as well. Further, the commutativity of \eqref{gap} and the definitions of the maps $\widetilde{\partial}^{\le(r)}$ \eqref{dtil} and $\beta_{\le r}^{\lle\prime}$ \eqref{brp} show that the map \eqref{gr2} is induced by the composition
\[
H^{\le r-1}\be(S_{\et},G\e)\overset{\delta^{(\lle r-1\lle)}}{\lra} H^{\le r}\be(S_{\et},R_{\e S^{\lle\prime}\!/\lbe S}^{\e(1)}\lbe(G\e))\overset{\left(\!\psi^{(\lbe r\lle)}\!\right)^{\!-1}}{\underset{\!\sim}{\lra}}H^{\le r}\be(S_{\fl},P_{\lbe S^{\lle\prime}\be/S}(G\e))\overset{\partial^{(\lle r\lle)}}{\lra}H^{\le r+1}\be(S_{\et},G\e),
\] 
where the maps $\delta^{(\lle r-1\lle)}$ and $\partial^{(\lle r\lle)}$ are the connecting homomorphisms appearing in the sequences \eqref{lse2} and \eqref{lse1}, respectively.

Now Proposition \ref{ak} yields an exact sequence of $2$-torsion abelian groups 
\begin{equation}\label{ko2}
0\to \krn\e \beta_{\lle r}\to \cok\e {\rm Cores}_{\e G}^{(r-1)}\overset{\!\be\gamma_{\lbe r}}{\to}\krn {\rm Res}_{\e G}^{(r+1)}\to \cok\e \beta_{\lle r}\to 0, 
\end{equation}
where $\gamma_{r}$ is the map \eqref{gr2} and $\beta_{\lle r}\colon \cok\e {\rm Res}_{\e G}^{(r)}\to\krn\e {\rm Cores}_{\e G}^{(r)}$ is the map \eqref{br}. Consequently, the following holds.

\begin{theorem}\label{boo}  Let $(\e f,G\e)$ be an admissible quadratic Galois pair (see Definition {\rm \ref{qadm}}). Then, for every integer $r\geq 0$, there exists a canonical exact sequence of abelian groups
\[
0\to \krn\e \gamma_{\lle r}\to \cok\e {\rm Res}_{\e G}^{(r)}\overset{\!\be\beta_{ r}}{\to}\krn\e {\rm Cores}_{\e G}^{(r)}\to \cok\e \gamma_{\lle r}\to 0, 
\]
where $\beta_{\lle r}$ is the map \eqref{br} and $\gamma_{\lle r}\colon \cok\e {\rm Cores}_{\e G}^{(r-1)}\to \krn {\rm Res}_{\e G}^{(r+1)}$ is the map \eqref{gr2}.
\end{theorem}

If $(\e f,G\e)$ is an admissible quadratic Galois pair with Galois group $\Delta=\{1,\tau\}$, then $\Delta$ acts naturally on $H^{\le r}\be(\spp_{\et}, G\e)$ and the definitions \eqref{bmor} and \eqref{t1} show that the map $\beta_{\lle r}\colon \cok\e {\rm Res}_{\e G}^{(r)}\to\krn\e {\rm Cores}_{\le G}^{(r)}$ \eqref{br} is induced by the map $1-\tau\colon H^{\le r}\be(\spp_{\et}, G\e)\to H^{\le r}\be(\spp_{\et}, G\e)$. Thus $\krn\le\beta_{\lle r}$ is the cokernel of the map $H^{\lle r}\be(S_{\et},G\e)\to H^{\le r}\be(\spp_{\et}, G\e)^{\Delta}$ induced by ${\rm Res}_{\le G}^{(r)}$. Further, $\cok\beta_{\lle r}=H^{\lbe -1}_{\lbe *}\be(\Delta,H^{\le r}\be(\spp_{\et}, G\e))$, where
\begin{equation}\label{h1s}
H^{\lbe -1}_{\lbe *}\be(\Delta,H^{\le r}\be(\spp_{\et}, G\e))={}_{N}\be H^{\le r}\be(\spp_{\et}, G\e)/(1-\tau)H^{\le r}\be(\spp_{\et}, G\e)
\end{equation}
with ${}_{N}\be H^{\le r}\be(\spp_{\et}, G\e)=\krn\e {\rm Cores}_{\le G}^{(r)}$ \eqref{eas2}. The choice of notation above is motivated by the fact that
\eqref{h1s} is naturaly isomorphic to a subgroup of the (Tate) $\Delta$-cohomology group $H^{\lbe -1}\be(\Delta,H^{\le r}\be(\spp_{\et}, G\e))={}_{(1+\tau)}H^{\le r}\be(\spp_{\et}, G\e)/(1-\tau)H^{\le r}\be(\spp_{\et}, G\e)$, where ${}_{(1+\tau)}H^{\le r}\be(\spp_{\et}, G\e)$ is the kernel of the endomorphism $1+\tau\colon H^{\le r}\be(\spp_{\et}, G\e)\to H^{\le r}\be(\spp_{\et}, G\e)$. Indeed, an application of Lemma \ref{ker-cok} to the commutative diagram
\[
\xymatrix{H^{\le r}\be(\spp_{\et}, G\e)\ar[drr]_{1+\tau}\ar[rr]^{{\rm Cores}_{\e G}^{(r)}}&& H^{\le r}\be(S_{\et}, G\e)\ar[d]^{{\rm Res}_{\e G}^{(r)}}\\
&&H^{\le r}\be(\spp_{\et}, G\e)
}
\]
yields an exact sequence of $2$-torsion abelian groups
\[
0\to H^{\lbe -1}_{\lbe *}\be(\Delta,H^{\le r}\be(\spp_{\et}, G\e))\to H^{\lbe -1}\be(\Delta,H^{\le r}\be(\spp_{\et}, G\e))\to \krn {\rm Res}_{\e G}^{(r)}.
\]
Note that $H^{\lbe -1}_{\lbe *}\be(\Delta,H^{\le 0}\lbe(\spp_{\et}, G\e))=H^{\lbe -1}\be(\Delta,H^{\le 0}\lbe(\spp_{\et}, G\e))$. Now \eqref{ko2} yields the following statement.

\begin{theorem} \label{main} Let $(\e f,G\e)$ be an admissible quadratic Galois pair with  Galois group $\Delta$ (see Definition {\rm \ref{qadm}}). Then, for every integer $r\geq 0$, there exists a canonical exact sequence of abelian groups
\[
\begin{array}{rcl}
H^{\lle r}\be(S_{\et},G\e)\to H^{\le r}\be(\spp_{\et}, G\e)^{\Delta}&\to& \cok\e {\rm Cores}_{\e G}^{(r-1)}\overset{\!\be\gamma_{ r}}{\to}\krn {\rm Res}_{G}^{(r+1)}\\
&\to& H^{\lbe -1}_{\lbe *}\be(\Delta,H^{\le r}\be(\spp_{\et}, G\e))\to 0,
\end{array}
\]
where $\gamma_{r}$ is the map \eqref{gr2} and $H^{\lbe -1}_{\lbe *}\be(\Delta,H^{\le r}\be(\spp_{\et}, G\e))$ is the group \eqref{h1s}.
\end{theorem}

\begin{remarks}\label{hs}\indent
\begin{enumerate}
\item[(a)] It seems likely that the exact sequence of Theorem \ref{main} (=Theorem \ref{main0}) can also be derived from the Hochschild-Serre spectral sequence associated to $\Delta$ and $G$
\begin{equation}\label{hsss}
H^{s}\lbe(\Delta,H^{\le r}\lbe(\spp_{\lbe\et}, G\e))\implies H^{s+r}\lbe(S_{\et}, G\e),
\end{equation}
where the hypotheses $|\Delta|=2$ and $G(2)=0$, i.e., $R_{\e S^{\lle\prime}\!/\lbe S}^{\e(1)}\lbe(G\le)_{2}=G_{\lbe 2}$, should simplify the analysis of the various associated intermediate exact sequences and commutative diagrams. Regarding other statements in this paper that apply to possibly ramified (non Galois) coverings of arbitrary rank (such as Theorem \ref{th1}), they may also follow from a (careful) analysis of a suitable spectral sequence (e.g., that in \cite[I, Theorem 3.4.4(i), p.~58]{t}) but, as noted in the Introduction, we doubt that such an approach will be significantly simpler than the one adopted in this paper.

\item[(b)] Let $f\colon\spp\to S$ be a cyclic Galois covering of constant rank $n\geq 2$ with Galois group $\Delta$ such that the pair $(\le f,G\le)$ is admissible (see Definition \ref{adm}). Then, via the periodicity isomorphisms of Tate $\Delta$-cohomology groups $H^{\le i}\lbe(\Delta,-)\isoto H^{\le i+2}\lbe(\Delta,-)$ \cite[Theorem 5, p.~108]{aw} and the equality $\cok\e {\rm Cores}_{\e G}^{(0)}=H^{\le 0}(\Delta,G(\spp\le))$, the exact sequence of terms of low degree \cite[p.~309, line 8]{mi1} associated to the Hochschild-Serre spectral sequence \eqref{hsss} corresponds to an exact sequence
\begin{equation}\label{ldg}
\begin{array}{rcl}
0\to H^{\le -1}\be(\Delta,G(\spp\le))&\to& H^{\lle 1}\be(S_{\et},G\e)\to H^{\le 1}\be(\spp_{\et}, G\e)^{\Delta}\to \cok\e {\rm Cores}_{\e G}^{(0)}\\
&\to&\krn {\rm Res}_{G}^{(2)}\to H^{\le -1}\be(\Delta,H^{\le 1}\be(\spp_{\et}, G\e))\to H^{\le -1}\be(\Delta,G(\spp\le)).
\end{array}
\end{equation}
Now, if $N_{\be\Delta}=\sum_{\e\sigma\in\Delta}\sigma$ is the norm element of $\Z[\e\Delta\e]$, then the following diagram commutes
\[
\xymatrix{H^{\le 1}\be(\spp_{\et}, G\e)\ar[drr]_{N_{\be\Delta}}\ar[rr]^{{\rm Cores}_{\e G}^{(1)}}&& H^{\le 1}\be(S_{\et}, G\e)\ar[d]^{{\rm Res}_{\e G}^{(1)}}\\
&&H^{\le 1}\be(\spp_{\et}, G\e),
}
\]
whence ${\rm Cores}_{\e G}^{(1)}$ induces a map $H^{\le -1}\be(\Delta,H^{\le 1}\be(\spp_{\et}, G\e))\to \krn\le{\rm Res}_{\le G}^{(1)}$. In view of \cite[Proposition 4.2]{bea}, it seems likely that the preceding map is the composition
\begin{equation}\label{iso}
H^{\le -1}\be(\Delta,H^{\le 1}\be(\spp_{\et}, G\e))\to H^{\le -1}\be(\Delta,G(\spp\le))\isoto \krn\le{\rm Res}_{\le G}^{(1)},
\end{equation}
where the first arrow is the last map in \eqref{ldg} and the second arrow is the isomorphism induced by the first nontrivial map in \eqref{ldg}. If this is the case, and $H^{\lbe -1}_{\lbe *}\be(\Delta,H^{\le 1}\be(\spp_{\et}, G\e))$ denotes ${}_{N}\be H^{\le 1}\be(\spp_{\et}, G\e)/I_{\lbe \Delta}H^{\le 1}\be(\spp_{\et}, G\e)$, where $I_{\lbe \Delta}$ is the augmentation ideal of $\Z[\Delta]$, then \eqref{ldg} induces an exact sequence of abelian groups
\[
\begin{array}{rcl}
H^{\lle 1}\be(S_{\et},G\e)\to H^{\le 1}\be(\spp_{\et}, G\e)^{\Delta}&\to& \cok\e {\rm Cores}_{\e G}^{(0)}\to\krn {\rm Res}_{G}^{(2)}\\
&\to& H^{\lbe -1}_{\lbe *}\be(\Delta,H^{\le 1}\be(\spp_{\et}, G\e))\to 0
\end{array}
\]
which might well be regarded as a generalization (to arbitrary cyclic Galois coverings) of the exact sequence of Theorem \ref{main} for $r=1$.
\end{enumerate} 
\end{remarks}

\section{Arithmetical applications}\label{arith}

Recall that a global field is either a number field, i.e., a finite extension of $\Q$, or a global function field, i.e., the function field of a smooth, projective and irreducible algebraic curve over a finite field. Let $K\be/\be F$ be a quadratic Galois extension of global fields with Galois group $\Delta$, let $\si$ be a nonempty finite set of primes of $F$ containing the archimedean primes and the non-archimedean primes that ramify in $K$ and let $\si_{K}$ denote the set of primes of $K$ lying above the primes in $\si$. If $v\in\si$ and there exists a prime $w$ of $K$ lying above $v$ such that $[K_{w}\colon\! F_{\be v}\le]=2$, where $K_{w}$ and $F_{\be v}$ denote the completions of the indicated fields at the indicated primes, then $w$ is the unique prime of $K$ lying above $v$ (for the archimedean case, see \cite[Remark 1.2.1, p.~10]{gras}). Set
\begin{equation}\label{sr}
\si^{\lle\prime}=\{ v\in\si\colon [K_{w}\colon\! F_{\be v}\le]=2\text{ for some (unique) }w|v \}
\end{equation}
and
\begin{equation}\label{sr2}
\rho={\rm max}\e\{0,\#\si^{\le\prime}\!-1\}.
\end{equation}
We will write $\si^{\lle\prime}_{\le\rm real}$ for the set of real primes of $F$ that lie in $\si^{\lle\prime}$, i.e., $\si^{\lle\prime}_{\le\rm real}$ is the set of real primes of $F$ that ramify in $K$. For every $v\in \si^{\lle\prime}$, we set $\Delta_{\e v}={\rm Gal}(K_{w}\lbe/\lbe F_{\be v})$, where $w$ is the unique prime of $K$ lying above $v$. We now let $\ofs$ and $\oks$ denote, respectively, the rings of $\si\le$-\le integers of $F$ and $\si_{K}\le$-\le integers of $K$. If $\spp=\spec\oks,S=\spec \ofs$ and $f\colon \spp\to S$ is the canonical morphism induced by the inclusion $\ofs\subset\oks$, then $f$ is a quadratic Galois covering of noetherian schemes.

\subsection{Ideal class groups} \label{une} 
Clearly, if $f\colon \spp\to S$ is as above, then $(\e f,\bg_{m,S}\lbe)$ is an admissible quadratic Galois pair (see Definition \ref{qadm}). The map $\gamma_{\le 1}$ \eqref{gr2} is a map
\begin{equation}\label{key}
\gamma_{\le 1}\colon \ofss\le/\be N_{\be K\be/\lbe F}\e\okss\to \br(\be\oks\lbe/\ofs\lbe),
\end{equation}
where $\br(\be\oks\lbe/\ofs\lbe)=\krn[\e\br f\colon \br\ofs \to \br\oks\e]$ is the relative Brauer group of $\spp\lbe/S\,$\footnote{ In the present setting the cohomological Brauer groups and the Brauer groups of equivalence of Azumaya algebras agree. See Remark \ref{brr}(a) below.}\e.

Now, by \cite[II, Proposition 2.1, p.~163]{adt} and \cite[Proposition 48(3), p.~159]{sha}, there exists a canonical exact and commutativite diagram of abelian groups
\begin{equation}\label{key1}
\xymatrix{
0\ar[r]& \ar[d]_{\br\lbe f}\br\e\ofs\ar[r]&\ar[d]^{\oplus\oplus\e{\rm Res}_{\lbe K_{\lbe w}\lbe /\lbe F_{\lbe v}}}\displaystyle{\bigoplus_{v\in\si}}\,\br\e F_{v}\ar[rr]^{ \sum{\rm inv}_{\lbe v}}&&\Q/\Z\ar[d]^{2}\\
0\ar[r]&\br\,\oks\ar[r]&\displaystyle{\bigoplus_{v\in\si}}\displaystyle{\bigoplus_{w\vert v}}\,\br\, K_{w}\ar[rr]^(.6){\sum \sum{\rm inv}_{\lbe w}}&&\Q/\Z.
}
\end{equation}
The above diagram yields an isomorphism of abelian groups
\begin{equation}\label{bis}
\begin{array}{rcl}
\br(\lbe\oks/\ofs)&\isoto&\krn\!\be\left[\,\displaystyle{\bigoplus_{v\in\si}}\e\br(\be K_{\lbe w}\lbe/\be F_{\be v})\overset{\!\Sigma\e{\rm inv}_{\lbe\lbe v}}{\lra}\Z/\lbe 2\le\Z\right]\\\\
&=&\krn\!\be\left[\,\displaystyle{\bigoplus_{v\in\si^{\le\prime}}}\e\br(\be K_{\lbe w}\lbe/\be F_{\be v})\overset{\!\Sigma\e{\rm inv}_{\lbe\lbe v}}{\lra}\Z/\lbe 2\le\Z\right],
\end{array}
\end{equation}
where $\si^{\le\prime}$ is the set \eqref{sr} and the map labeled $\Sigma\e{\rm inv}_{\lbe\lbe v}$ above is induced by $\sum{\rm inv}_{\lbe v}\colon \bigoplus_{v\in\si}\br\e F_{v}\to\Q/\Z$. Next we observe that, for every $v\in \si^{\le\prime}$, there exist canonical isomorphisms of abelian groups
\begin{equation}\label{bis2}
F_{\! v}^{\e *}\be/\be N_{\be K_{\lbe w}\lbe/\lbe F_{\be v}}K_{\lbe w}^{*}=H^{\le 0}\lbe(\lbe \Delta_{\e v},K_{\lbe w}^{*})\isoto H^{\le 2}\lbe(\lbe \Delta_{\e v},K_{\lbe w}^{*})\isoto \br(\be K_{\lbe w}\lbe/\be F_{\be v})\isoto \Z/\lbe 2\le\Z.
\end{equation}
See \cite[Proposition 5, p.~101; Theorem 5, p.~108]{aw} and \cite[Theorem 37, p.~155]{sha}. The first (respectively, second, third) isomorphism in \eqref{bis2} is induced by cup-product with the nontrivial element of $H^{\le 2}\lbe(\lbe \Delta_{\e v},\Z\le)$ (respectively, inflation map, invariant map ${\rm inv}_{\lbe\lbe v}$) and \eqref{bis} shows that $\br(\lbe\oks/\ofs)$ is canonically isomorphic to the kernel of the summation map $\bigoplus_{v\in\si^{\le\prime}}\be\Z/\lbe 2\le\Z\to\Z/\lbe 2\le\Z$. Thus $\br(\lbe\oks/\ofs)$ is a finite $2$-torsion abelian group of order
\begin{equation}\label{ord}
|\br(\lbe\oks/\ofs)|=2^{\le \rho},
\end{equation}
where $\rho\geq 0$ is the integer \eqref{sr2}. Next, it follows from the description of the map $\gamma_{\le 1}$ \eqref{key} (i.e., \eqref{gr2}) and the compatibility of the connecting homomorphisms  $\delta^{(\lle 0\lle)}$ \eqref{lse2} and $\partial^{(\lle 1\lle)}$ \eqref{lse1} with the left-hand horizontal maps in \eqref{key1}, the local cup-product, inflation map and invariant maps ${\rm inv}_{\lbe\lbe v}$ that the following diagram of abelian groups commutes
\begin{equation}\label{key2}
\xymatrix{\ofss/\lbe N_{\be K\be/\lbe F}\e\okss\ar[d]_{\gamma_{\le 1}}\ar[r]^(.33){\lambda}&\krn\!\be\left[\,\displaystyle{\bigoplus_{v\in\si}}\e F_{\! v}^{\e *}\be/\be N_{\be K_{\lbe w}\lbe/\lbe F_{\lbe v}}K_{\lbe w}^{*}\overset{\!\Sigma\e{\rm ord}_{v}}{\lra}\Z/\lbe 2\le\Z \right]\ar[d]^{\sim}\\
\br(\lbe\okss/\ofs)\ar[ur]^(.4){\sim}\ar[r]^(.34){\sim}&\krn\!\be\left[\,\displaystyle{\bigoplus_{v\in\si}}\e\br(\be K_{\lbe w}\lbe/\be F_{\be v})\overset{\!\Sigma\e{\rm inv}_{\lbe\lbe v}}{\lra}\Z/\lbe 2\le\Z \right],
}
\end{equation}
where $\gamma_{\le 1}$ is the map \eqref{key}, $\lambda$ is the canonical localization map, the right-hand vertical isomorphism comes from \eqref{bis2}, the bottom horizontal arrow is the map \eqref{bis} and the diagonal map is defined by the commutativity of the lower triangle. We conclude that
\begin{equation}\label{kg}
\krn \gamma_{\le 1}=\krn \lambda=W_{\!F,\e \si}/\lbe N_{\be K\be/\lbe F}\e\okss,
\end{equation}
where 
\begin{equation}\label{kg3}
W_{\!F,\e \si}=\{\varepsilon\in \ofss\colon \varepsilon\in N_{\be K_{\lbe w}\lbe/\lbe F_{\lbe v}}K_{\lbe w}^{*}\text{ for every $v\in\si\lle$}\}
\end{equation}
and the diagonal map in \eqref{key2} induces an isomorphism of abelian groups
\begin{equation}\label{kg4}
\cok\gamma_{\le 1}\isoto \cok \lambda.
\end{equation}
We now observe that $\img\le\lambda\simeq \ofss/\le W_{\!F,\e \si}$ is a finite $2$-torsion abelian group of order
\begin{equation}\label{isol}
2^{\le e}=[\ofss\colon W_{\!F,\e \si}],
\end{equation}
where $0\leq e\leq \rho$. Thus, by \eqref{ord} and \eqref{key2}, $\cok \lambda$ is finite $2$-torsion abelian group of order
\begin{equation}\label{isol2}
|\cok \lambda|=2^{\e \rho\le -\le e}.
\end{equation}

Now let $C_{F,\e\si}$ and $C_{\be K,\e\si}$ denote, respectively, the $\si$-ideal class group of $F$ and the $\si_{K}$-ideal class group of $K$, and recall the classical {\it $\si$-capitulation map}
\begin{equation}\label{cap}
j_{ K\be/\lbe F,\e\si}\colon C_{F,\e\si}\to C_{\lbe K,\e\si},
\end{equation}
i.e., the map induced by extending ideals from $F$ to $K$. The {\it relative $\si$-ideal class group of $K\!/\be F$} is the group
\begin{equation}\label{rnm}
C_{K/F,\e\si}=\krn\!\be\left[C_{\lbe K,\e\si}\overset{\!N_{\lbe K\!/\be F}}{\lra} C_{F,\e\si}\right],
\end{equation}
where $N_{\be K\be/\lbe F}$ is induced by taking $K\be/\be F$-norms of ideals in $\oks$.
The identifications $C_{F,\e\si}=H^{\le 1}(S_{\et},\bg_{m})$ and $C_{\lbe K,\e\si}=H^{\le 1}(\spp_{\et},\bg_{m})$ induce identifications
\begin{eqnarray*}
\cok\e j_{ K\be/\lbe F,\e\si}&=&\cok\e{\rm Res}_{\e \bg_{m,\lbe S}}^{(1)}\\
C_{K/F,\e\si}&=&\krn\e{\rm Cores}_{\e \bg_{m,\lbe S}}^{(1)}.
\end{eqnarray*}
Under the isomorphisms \eqref{kg} and \eqref{kg4}, the formula \eqref{isol2} and the preceding identifications, Theorem \ref{boo} for $r=1$ yields the following statement.

\begin{theorem}\label{nice} Let $K\be/\be F$ be a quadratic Galois extension of global fields and let $\si$ be a nonempty finite set of primes of $F$ containing the archimedean primes and the non-archimedean primes that ramify in $K$. Then there exists a canonical exact sequence of abelian groups
\[
0\to W_{\!F,\e \si}/\lbe N_{\be K\be/\lbe F}\e\okss\to \cok\e j_{ K\be/\lbe F,\e\si}\to C_{K/F,\e\si}\to \cok \lambda\to 0,
\]
where $W_{\!F,\e \si}\subseteq\ofss$ is given by \eqref{kg3}, $j_{ K\be/\lbe F,\e\si}$ is the $\si$-capitulation map \eqref{cap}, $C_{K/F,\e\si}$ is the relative $\si$-ideal class group of $K\!/\be F$ and
\[
\lambda\colon\ofss/\lbe N_{\be K\be/\lbe F}\e\okss\to\krn\!\be\left[\,\displaystyle{\bigoplus_{v\in\si}}\e F_{\! v}^{\e *}\be/\be N_{\be K_{\lbe w}\lbe/\lbe F_{\lbe v}}K_{\lbe w}^{*}\overset{\!\Sigma\e{\rm ord}_{v}}{\lra}\Z/\lbe 2\le\Z \right]
\]
is the canonical localization map. In particular, the number of ideal classes in $C_{\lbe K,\e\si}$ that do not lie in the image of $j_{ K\be/\lbe F,\e\si}$ is at least $[\e W_{\!F,\e \si}\colon\! N_{\be K\be/\lbe F}\e\okss\e]$ and the relative $\si$-ideal class number $|C_{K/F,\e\si}|$ is divisible by $2^{\e \rho\le -\le e}$, where the integers  $\rho$ and $e$ are given by \eqref{sr2} and \eqref{isol}, respectively.
\end{theorem}

\smallskip

\begin{remark} Let $j_{ K\be/\lbe F,\e\si}^{\Delta}\colon C_{F,\e\si}\to C_{\lbe K,\e\si}^{\Delta}$ be the $\Delta$-invariant capitulation map. Then \eqref{kg}, \eqref{kg4} and Theorem \ref{main} for $r=1$ show that there exist canonical isomorphisms of finite $2$-torsion abelian groups
\[
\begin{array}{rcl}
\cok j_{ K\be/\lbe F,\e\si}^{\Delta} &\simeq&W_{\!F,\e \si}/\lbe N_{\be K\be/\lbe F}\e\okss\\
H^{\lbe -1}\be(\Delta,C_{\lbe K,\e\si})&\simeq&\cok\lambda.
\end{array}
\]
Via periodicity \cite[Theorem 5, p.~108]{aw}, the first isomorphism above is, in fact, a particular case of \cite[Theorem 7.1]{ga07}.
\end{remark}

\subsection{N\'eron-Raynaud class groups of tori} \label{una} If $T$ is an algebraic torus over $F$, the {\it N\'eron-Raynaud $\Sigma$-class group of $T$} is the quotient
\begin{equation}\label{clgp}
C_{\e T\lbe,\e F,\e \Sigma}=\mathcal T^{\e 0}\lbe(\mathbb A_{S})/\e T(F\le)\e \mathcal T^{\e 0}\lbe(\lbe\mathbb A_{S}\lbe (S\e)),
\end{equation}
where $\mathbb A_{S}$ (respectively, $\mathbb A_{S}\lbe (S\le)$) is the ring of adeles (respectively, integral adeles) of $S$ and $\mathcal T^{\e 0}$ is the identity component of the N\'eron-Raynaud model of $T$ over $S$. The above group is finite \cite[\S1.3]{c} and coincides with the $\Sigma$-ideal class group of $F$ discussed in the previous subsection when $T=\bg_{m,\le F}$. We will write $C_{\le T\lbe,\e K,\e \Sigma}$ for the 
N\'eron-Raynaud $\Sigma_{K}$-class group of $T_{\be K}$. The {\it $\Delta$-invariant  capitulation map} $j_{\le T,\e K/F,\e\si}^{\Delta}\colon 
C_{\e T\lbe,\e F,\e \Sigma}\to C_{\le T\lbe,\e K,\e \Sigma}^{\le\Delta}$ was discussed in \cite[\S\S 4 and 5]{ga10} for arbitrary $F$-tori $T$ and arbitrary finite Galois extensions $K/F$ with Galois group $\Delta$. In this subsection we apply Theorem \ref{main} to obtain additional information on the above map for certain types of $F$-tori in the current setting, i.e., $K\lbe/\lbe F$ is a quadratic Galois extension. By \cite[(3.10)]{ga12}, there exists a canonical exact sequence of abelian groups 
\begin{equation}\label{seq1}
0\to C_{\e T\lbe,\e F,\e \Sigma}\to H^{1}(S_{\et}, \mathcal T^{\e 0}\e)\to \sha^{1}_{\Sigma}(F,T\e)\to 0,
\end{equation}
where $\sha^{1}_{\Sigma}(F,T\e)=\krn\!\be\left[H^{1}\be(F,T\e)\!\to\! \prod_{v\notin\Sigma}\!H^{1}\be(F_{v},T\e)\right]$ is the $\Sigma$-Tate-Shafarevich group of $T$ over $F$. We now assume that $T$ is an {\it invertible $F$-torus}, i.e., $T$ is isomorphic to a direct factor of an $F$-torus of the form $R_{\le E\lbe/\lbe F}(\bg_{m,\le E})$, where $E/F$ is a finite \'etale $F$-algebra. Then $H^{1}\be(F,T\e)=0$ by \cite[Lemma 4.8(a)]{ga10} and \eqref{seq1} yields an isomorphism of finite abelian groups
\begin{equation}\label{iso2}
C_{\e T\lbe,\e F,\e \Sigma}\isoto H^{1}(S_{\et}, \mathcal T^{\e 0}\e)
\end{equation}
Similarly, $T_{K}$ is an invertible $K$-torus and
$\mathcal T^{\e 0}_{S^{\prime}}$ is the identity component of the N\'eron-Raynaud model of $T_{K}$ over $\spp$ by \cite[\S7.2, Theorem 1(ii), p.~176]{blr} and \cite[${\rm VI_B}$, Proposition 3.3]{sga3}, whence the sequence \eqref{seq1} over $K$ yields an isomorphism of finite abelian groups $C_{\le T\lbe,\e K,\e \Sigma}\isoto H^{\le 1}\be(\spp_{\et}, \mathcal T^{\e 0}\lle)$.

\begin{theorem} Let $K\be/\be F$ be a quadratic Galois extension of global fields with Galois group $\Delta=\{1,\tau\}$, let $\si$ be a nonempty finite set of primes of $F$ containing the archimedean primes and the non-archimedean primes that ramify in $K$ and set $S=\spec \mathcal O_{\lbe F,\e \Sigma}$ and $\spp =\mathcal O_{\lbe K,\e \Sigma_{K}}$. If $T$ is an invertible $F$-torus such that $\mathcal T^{\e 0}_{\be k(s)}$ is a $k(s)$-torus for every point $s\in S$ with ${\rm char}\, k(s)=2$, where $\mathcal T^{\e 0}$ is the identity component of the N\'eron-Raynaud model of $T$ over $S$, then there exists a canonical exact sequence of abelian groups
\[
\begin{array}{rcl}	
0&\to& H^{-1}\lbe(\Delta, \mathcal T^{\e 0}(\spp\e))\to C_{\e T\lbe,\e F,\e \Sigma}\to C_{\le T\lbe,\e K,\e \Sigma}^{\le\Delta}\to \mathcal T^{\e 0}(S\e)/N_{\be \spp\lbe/\lbe S}\mathcal T^{\e 0}(\spp\le)\\\\
&\to&\krn[H^{2}(S_{\et},\mathcal T^{\e 0})\to H^{2}(\spp_{\et},\mathcal T^{\e 0}\e)]\to {}_{N}C_{\le T\lbe,\e K,\e \Sigma}/(1\!-\!\tau)C_{\le T\lbe,\e K,\e \Sigma}\to 0
\end{array}
\]
\end{theorem}
\begin{proof} The hypotheses imply that $(\le f,\mathcal T^{\e 0}\le)$ is an admissible quadratic Galois pair (see Definition \ref{qadm}). The theorem then follows by applying Theorem \ref{main} with $r=1$  and $(\le f, G\le)=(\le f,\mathcal T^{\e 0}\le)$ using the definition \eqref{h1s}, the isomorphism in \eqref{iso} and the isomorphisms \eqref{iso2} over $F$ and over $K$.
\end{proof}

\subsection{Abelian schemes} In order to be able to appeal to the results in \cite[II, \S5]{adt}, in this subsection we assume that $2$ is invertible on $S$, i.e., $\si$ contains all primes $v$ of $F$ such that ${\rm char}\e k(v)=2$ (in particular ${\rm char}\e F\neq 2$ since $\si$ is finite). Let $\mathcal A$ be an abelian scheme over $S$ with generic fiber $A$. We will write $A^{t}$ for the dual abelian variety of $A$. If $v$ is a real prime of $F$, let $\pi_{0}(A^{\lbe t}\lbe(F_{\be v}))$ be the group of connected components of $A^{\lbe t}\lbe(F_{\be v})$ and set
\begin{equation}\label{pio}
\pi_{0}(\lbe A(F_{\be v}))^{\le\prime}={\rm Hom}(\pi_{0}(\lbe A^{\lbe t}\lbe(F_{\be v})),\Z/2\le).
\end{equation}
Note that, since $\mathcal A$ is a N\'eron model of $A$ over $S$ \cite[\S1.2, Proposition 8, p.~15]{blr}, we have $\mathcal A(S\e)=A(F\e)$. Now let  $\sha^{1}_{\Sigma}(F,A\e)\!=\krn\!\be\left[H^{1}\be(F,A\e)\be\to\be \prod_{v\notin\Sigma}\be H^{1}\be(F_{\be v},A\e)\right]$ be the $\Sigma$-Tate-Shafarevich group of $A$ over $F$. By \cite[II, Proposition 5.1(a), p.~197, and sequence (5.5.1), p.~ 201]{adt}, the canonical map $H^{\le 1}\be(S_{\et},\mathcal A)\to H^{1}(F,A)$ induces an isomorphism of torsion abelian groups
\begin{equation}\label{tor1}
H^{\le 1}\be(S_{\et},\mathcal A)\isoto\sha^{1}_{\Sigma}(F,A\e).
\end{equation}
Similarly, $H^{1}(\spp_{\et},\mathcal A)\isoto \sha^{1}_{\Sigma}(K,A\e)$, where $\sha^{1}_{\Sigma}(K,A\e)$ is the $\Sigma_{K}$-Tate-Shafarevich group of $A_{K}=A\times_{\lbe F} K$ over $K$.
In \cite[\S6]{ga18b} we discussed the group
\[
\begin{array}{rcl}
\krn[\e\sha^{1}_{\Sigma}(F,A\le)\to \sha^{1}_{\Sigma}(K,A\le)\e]&\simeq&\krn[ H^{1}(S,\mathcal A)\overset{\!{\rm Res}_{\!\mathcal A}^{(1)}}{\lra} H^{1}(\spp_{\et},\mathcal A\e)]\\
&\simeq&H^{-1}(\Delta, A(K))
\end{array}
\]
(see \eqref{ldg} for the second isomorphism). In this subsection we discuss the kernel and cokernel of the map ${\rm Res}_{\be\mathcal A}^{(r),\lle \Delta}\colon H^{r}(S,\mathcal A)\to H^{r}(\spp_{\et},\mathcal A\e)^{\Delta}$ for $r=1$ and $2$.

\begin{lemma} \label{prea} Assume that $2$ is invertible on $S$. Then there exists a canonical isomorphism of $2$-torsion abelian groups
\[
\krn\e{\rm Res}_{\lbe\mathcal A}^{(3)}=\displaystyle
\prod_{\e v\lle\in\lle\si^{\lle\prime}_{\rm real}}\!\pi_{0}(\lbe A(F_{\be v}))^{\e\prime} ,
\]
where, for each $v\lle\in\lle\si^{\lle\prime}_{\rm real}$, $\pi_{0}(\lbe A(F_{\be v}))^{\e\prime}$ is the group \eqref{pio}.
\end{lemma}
\begin{proof} By Remark \ref{ntor}(a) and \cite[Proposition II.5.1(a), p.~197]{adt}, we have
\[
\begin{array}{rcl}
\krn\e{\rm Res}_{\lbe\mathcal A}^{(3)}&=&\krn[H^{\le 3}\lbe(S_{\et},\mathcal A)_{2}\to H^{\le 3}\lbe(\spp_{\et},\mathcal A)_{2}\lbe]\\
&\simeq&{\displaystyle\prod_{\e \text{$v$ real}}}\krn[H^{\lle 3}\lbe(F_{\be v}, A)\to \prod_{\e w|v}\!H^{\lle 3}\lbe(K_{w}, A)\lbe]=\displaystyle\prod_{v\e\in\e\si^{\prime}_{\lle\rm real}}\!\!H^{\lle 3}\lbe(F_{\be v}, A\le).
\end{array}
\]
The proof is now completed by \cite[I, Remark 3.7, p.~46]{adt}, which shows that, for every $v\e\in\e\si^{\le\prime}_{\lle\rm real}$, there exists a canonical isomorphism of $2$-torsion abelian groups
$H^{\le 3}\lbe(F_{\be v}, A\le)\simeq\pi_{0}(\lbe A(F_{\be v}))^{\e\prime}$.
\end{proof}

\begin{theorem}\label{abv} Let $F$ be a global field of characteristic different from $2$, $K\be/\be F$ a quadratic Galois extension with Galois group $\Delta$ and $\si$ a nonempty finite set of primes of $F$ containing the archimedean primes, the non-archimedean primes that ramify in $K$ and all primes $v$ such that ${\rm char}\e k(v)=2$. Set $S=\spec \mathcal O_{\lbe F,\e \Sigma}$ and $\spp =\mathcal O_{\lbe K,\e \Sigma_{K}}$ and let $\mathcal A$ be an abelian scheme over $S$ with generic fiber $A$. Then there exist canonical exact sequences of torsion abelian groups	
\begin{enumerate}
\item[(i)]
\[
\begin{array}{rcl}	
0&\to& H^{-1}\lbe(\Delta, A(K))\to\sha^{1}_{\Sigma}(F,A\e)\to \sha^{1}_{\Sigma}(K,A\e)^{\Delta}\to A(F\e)/N_{\be K\lbe/\lbe F}A(K)\\\\
&\to&\krn[H^{2}(S,\mathcal A)\to H^{2}(\spp_{\et},\mathcal A\e)]\to {}_{N}\sha^{1}_{\Sigma}(K,A\e)/(1\!-\!\tau)\sha^{1}_{\Sigma}(K,A\e)\to 0,
\end{array}
\]
where $\tau$ is the nontrivial element in $\Delta$, and
\item[(ii)]
\[
\begin{array}{rcl}	
H^{2}(S,\mathcal A)&\to& H^{2}(\spp_{\et},\mathcal A\e)^{\Delta}\to
\sha^{1}_{\Sigma}(F,A\e)/N_{\be K\lbe/\lbe F}\le\sha^{1}_{\Sigma}(K,A\e)\\
&\to&\displaystyle\prod_{\e v\lle\in\lle\si^{\lle\prime}_{\rm real}}\!\!\pi_{0}(\lbe A(F_{\be v}))^{\e\prime}\to H^{\lbe -1}_{\lbe *}\be(\Delta,H^{\le 2}\lbe(\spp_{\et}, \mathcal A\e))\to 0,
\end{array}
\]
where, for each $v\lle\in\lle\si^{\prime}_{\rm real}$, $\pi_{0}(\lbe A(F_{\be v}))^{\e\prime}$ is the group \eqref{pio}.
\end{enumerate}
\end{theorem}
\begin{proof} The theorem is immediate from Theorem \ref{main} applied to $r=1$ and $2$ and $(\le f,G\le)=(\le f,\mathcal A\e)$ using Lemma \ref{prea} and the canonical isomorphisms \eqref{tor1} over $F$ and over $K$.	
\end{proof}

\section{Applications to the Brauer group}\label{bra}

If $f\colon \spp\to S$ is a quadratic Galois covering of locally noetherian schemes, then $(\e f,\bg_{m,\le S}\lbe)$ is an admissible quadratic Galois pair (see Definition \ref{qadm}). The following proposition combines Theorem \ref{boo} for $r=2$ and Theorem \ref{main} for $r=1$ and $r=2$ when $(\e f,G\e)=(\e f,\bg_{m,\le S}\lbe)$ using the notations \eqref{eas1}-\eqref{eas3}.

\begin{theorem}\label{gat} Let $f\colon \spp\to S$ be a quadratic Galois covering of locally noetherian schemes with Galois group $\Delta$. Then there exist canonical exact sequences of abelian groups
\begin{enumerate}
\item[(i)] 
\[
\begin{array}{rcl}
\pic S\!&\!\to& (\pic\spp\le)^{\Delta}\to\varGamma( S,\mathcal O_{\lbe S}\lbe)^{*}\be/N_{\spp\!\lbe/\lbe S}\varGamma(\spp\!,\mathcal O_{\lbe S^{\prime}}\lbe)^{*}\\
&\overset{\!\!\gamma_{1}}{\to}&\brp\lbe(\spp\be/\lbe S\e)\to H^{\le -1}_{\lbe *}\lbe(\Delta,\pic \spp\lle)\to 0,
\end{array}
\]
where $\gamma_{\lle 1}$ is the map \eqref{gr2} and $H^{\le -1}_{\lbe *}\lbe(\Delta,\pic \spp\lle)$ is the group \eqref{h1s} associated to $(r, G\e)=(1,\bg_{m,\le S}\le)$ and $\brp\lbe(\spp\be/\lbe S\e)$ is the relative cohomological Brauer group \eqref{relb},
\item[(ii)] 
\[
\begin{array}{rcl}
\brp S\!&\!\to& (\brp\spp\le)^{\Delta}\to\pic S/N_{\spp\!\lbe/\lbe S}\le\pic\spp\\
&\overset{\!\!\gamma_{2}}{\to}&\krn[H^{\lle 3}\lbe(S_{\et},\bg_{m})\!\to\! H^{\le 3}\lbe(\spp_{\et}, \bg_{m})]\to H^{\le -1}_{\lbe *}\lbe(\Delta,\brp\spp\lle)\to 0,
\end{array}
\]
where $\gamma_{\lle 2}$ is the map \eqref{gr2} and $H^{\le -1}_{\lbe *}\lbe(\Delta,\brp\spp\lle)$ is the group \eqref{h1s} associated to $(r, G\e)=(2,\bg_{m,\le S}\le)$, and
\item[(iii)]
\[
0\to \krn\e \gamma_{\lle 2}\to \brp\lle \spp\!/\e{\rm Res}_{\le\spp\!/\lbe S}\e\brp\lle S\overset{\!\!\lbe\beta_{ 2}}{\to}{}_{N}\brp\lle\spp\to \cok\e \gamma_{\lle 2}\to 0,
\]
where $\beta_{ 2}$ and $\gamma_{\lle 2}$ are the maps \eqref{brp} and \eqref{gr2}, respectively, for $(r, G\e)=(2,\bg_{m,\le S}\le)$.
\end{enumerate}
\end{theorem}

\begin{corollary}\label{cor0} Let $\spp\to S$ be a quadratic Galois covering of locally noetherian schemes with Galois group $\Delta=\{1,\tau \}$. 
\begin{enumerate}
\item[(i)] Assume that the map \eqref{gr2} associated to $(r, G\e)=(1,\bg_{m,\le S}\le)$
\[
\gamma_{\le 1}\colon \varGamma( S,\mathcal O_{\lbe S}\lbe)^{*}\be/N_{\spp\!\lbe/\lbe S}\varGamma(\spp\!,\mathcal O_{\lbe S^{\prime}}\lbe)^{*}\to\brp\lbe(\spp\be/\lbe S\e) 
\]
is the zero map. Then there exists a canonical exact sequence of abelian groups
\[
\begin{array}{rcl}
0&\to& H^{\le -1}_{\lbe *}\lbe(\Delta,\pic \spp\le)\to \brp S\to (\brp\spp\e)^{\Delta}
\to\pic S/N_{\spp\!\lbe/\lbe S}\le\pic\spp\\
&\to&\krn[H^{\lle 3}\lbe(S_{\et},\bg_{m})\to H^{\le 3}\lbe(\spp_{\et}, \bg_{m})]\to H^{\le -1}_{\lbe *}\lbe(\Delta,\brp\spp\le)\to 0.
\end{array}
\]
\item[(ii)] Assume that the map \eqref{gr2} associated to $(r, G\e)=(2,\bg_{m,\le S}\le)$
\[
\gamma_{\le 2}\colon \pic S/N_{\spp\!\lbe/\lbe S}\le\pic\spp\to\krn[H^{\lle 3}\lbe(S_{\et},\bg_{m})\!\to\! H^{\le 3}\lbe(\spp_{\et}, \bg_{m})]
\]
is the zero map. Then there exist canonical exact sequences of abelian groups
\begin{enumerate}
\item[(a)]
\[
\brp S\to (\brp\spp\le)^{\Delta}\to\pic S/N_{\spp\!\lbe/\lbe S}\le\pic\spp\to 0,
\]
\item[(b)]
\[
\begin{array}{rcl}
0\to\pic S/\le N_{\spp\!\lbe/\lbe S}\le\pic\spp&\to& \brp\lle \spp/\le {\rm Res}_{\le\spp\!/\lbe S}\e\brp\lle S\to {}_{N}\brp\lle\spp\\
&\to& H^{\lle 3}\lbe(S_{\et},\bg_{m})\to H^{\le 3}\lbe(\spp_{\et}, \bg_{m})
\end{array}
\]
and
\item[(c)]
\[
0\to\brp\lle \spp/(\brp\spp\le)^{\Delta}\to {}_{N}\brp\lle\spp\to H^{\lle 3}\lbe(S_{\et},\bg_{m})\!\to\! H^{\le 3}\lbe(\spp_{\et}, \bg_{m}).
\]
\end{enumerate}
\end{enumerate}	
\end{corollary}

\begin{remarks} \label{brr}\indent
\begin{enumerate}
\item[(a)] If there exists an ample invertible sheaf on $S$ and $S$ is noetherian and separated, then the canonical map $\br S\hookrightarrow(\brp S)_{\rm tors}$ is an isomorphism of torsion abelian groups \cite{dJ}. If, in addition, the strictly local rings \footnote{ I.e., the strict henselizations of the local rings.} of $S$ are factorial, then $\brp S$ is a torsion abelian group \cite[II, Theorem 1.4, p.~71]{dix}. Thus, if $S$ and $\spp$ admit ample invertible sheaves, have factorial strictly local rings and are noetherian and separated, then there exist analogs of Theorem \ref{gat} and Corollary \ref{cor0} where the cohomological Brauer groups $\brp S$ and $\brp \spp$ are replaced with the Brauer groups $\br S$ and $\br \spp$ of equivalence classes of Azumaya algebras on $S$ and $\spp$, respectively.
\item[(b)] The hypothesis in assertion (i) of Corollary \ref{cor0} certainly holds if the norm map $N_{\spp\!\lbe/\lbe S}\colon \varGamma(\be \spp\be,\be\mathcal O_{\lbe S^{\prime}}\lbe)^{*}\to \varGamma(\be S,\mathcal O_{\lbe S}\lbe)^{*}$ is surjective. This is the case, for example, if $S$ and $\spp$ admit proper and flat morphisms to a locally noetherian scheme $T$ with $\varGamma(\lbe T,\mathcal O_{\le T}\lbe)^{*}=(\varGamma(\lbe T,\mathcal O_{\le T}\lbe)^{*})^{2}$ whose geometric fibers are reduced and connected. Indeed, 
by \cite[Exer. 3.11, pp.~25 and 64\e]{klei}, $\varGamma(\lbe \spp\be,\be\mathcal O_{\lbe S^{\prime}}\lbe)^{*}=\varGamma(\lbe S,\mathcal O_{\lbe S}\lbe)^{*}=\varGamma(\le T,\mathcal O_{\lbe T}\lbe)^{*}$, whence $\Delta$ acts trivially on $\varGamma(\lbe \spp\be,\be\mathcal O_{\lbe S^{\prime}}\lbe)^{*}$ and $N_{\spp\!\lbe/\lbe S}\varGamma(\be \spp\be,\be\mathcal O_{\lbe S^{\prime}}\lbe)^{*}=(\varGamma(\lbe T,\mathcal O_{\le T}\lbe)^{*})^{2}=\varGamma(\lbe T,\mathcal O_{\le T}\lbe)^{*}=\varGamma(\lbe S,\mathcal O_{\lbe S}\lbe)^{*}$. In particular, Corollary \ref{cor0}(i) applies to quadratic Galois coverings of proper and geometrically integral schemes over a quadratically closed field (e.g., an algebraically closed field or a separably closed field of characteristic different from $2$).
\item[(c)] The hypothesis in assertion (ii) of Corollary \ref{cor0} obviously holds if either $N_{\spp\!\lbe/\lbe S}\le\pic\spp=\pic S$ or $H^{\lle 3}\lbe(S_{\et},\bg_{m})\!\to\! H^{\le 3}\lbe(\spp_{\et}, \bg_{m})$ is injective. The preceding map is injective in the following cases.
\begin{itemize}
\item $S$ is a smooth and affine curve over a $p$-adic field. Indeed, in this case  $H^{\lle 3}\lbe(S_{\et},\bg_{m})=0$ by \cite[Corollary 4.10]{svh}.
\item $2$ is invertible on $S$ and ${\rm cd}_{\lle 2}(S_{\et})\leq 2$. See Remark \ref{ntor}(b). The indicated conditions hold, for example, if $S$ is an affine $2$-dimensional scheme of finite type over a separably closed field of characteristic different from $2$ \cite[XIV, Corollary 3.2]{sga4}.
\end{itemize}
\end{enumerate}	
\end{remarks}

\begin{examples}\label{ref} The following examples were suggested by the referee.

Let $k$ be an algebraically closed field of characteristic different from $2$, set $g_{\le 1}(x,y)=(x-a_{\le 1})\dots(x-a_{\le n})(xy-1)\in k[x,y]$, where $n\geq 1$ and $a_{i}\neq 0$ for every $i$, and let $g_{\le 2}(x,y)=(x+a_{\le 1}y)\dots(x+a_{\le n}y)(x-1)\in k[x,y]$, where $n\equiv0\pmod 2$. For $g=g_{\le 1}$ or $g_{\le 2}$, set $A=k[x,y][\e g(x,y)^{-1}]$, $B=A[\e\sqrt{g(x,y)}\e]$, $S=\spec A$ and $\spp=\spec B$. Then the canonical morphism $\spp\to S$ is a quadratic Galois covering of noetherian schemes and the exact sequence of Theorem \ref{gat}(i) reduces to an exact sequence of abelian groups
\[
\hskip 2cm 0\to\frac{\varGamma( S,\mathcal O_{\lbe S}\lbe)^{*}}{N_{\spp\!\lbe/\lbe S}\varGamma(\spp\!,\mathcal O_{\lbe S^{\prime}}\lbe)^{*}}\overset{\!\be\gamma_{1}}{\lra}\brp\lbe(\spp\be/\lbe S\e)\to H^{\le -1}_{\lbe *}\lbe(\Delta,\pic \spp\lle)\to 0
\]
which is isomorphic to the canonical exact sequence of finite $2$-torsion abelian groups
\[
0\to (\Z/2)^{n}\to(\Z/2)^{n+1}\overset{\!\be \pi}{\to}\Z/2\to 0,
\]
where $\pi$ is the canonical projection onto the last factor. See \cite[p.~3282, line 2, Theorem 2.2(b) and Propositions 3.1 and 3.4(b)]{f13} for $g=g_{\le 1}$ and \cite[Section 3.4, pp.~37-42 and Proposition 3.7(b), p.~29]{bul} for $g=g_{\le 2}$. In particular, $\gamma_{1}$ is {\it not} the zero map and therefore the hypothesis (as well as the conclusion) of Corollary \ref{cor0}(i) fails. On the other hand, the hypothesis (and therefore the conclusion) of Corollary \ref{cor0}(ii) holds since $\pic S=0$. When $g=g_{\le 1}$, \cite[Proposition 3.5 and its proof]{f13} shows that ${\rm Res}_{\le\spp\!/\lbe S}\colon \brp\lle S\to\brp\lle\spp$ is surjective, $\Delta$ acts trivially on $\brp\spp$ and $(\brp\lle\spp\e)_{2}\simeq(\Z/2)^{n+1}$, whence the exact sequences (a)-(c) in Corollary \ref{cor0}(ii) yield isomorphisms of abelian groups
\[
\krn[H^{\lle 3}\lbe(S_{\et},\bg_{m})\to H^{\le 3}\lbe(\spp_{\et}, \bg_{m})]\simeq (\brp\lle\spp\e)_{2}\simeq(\Z/2)^{n+1}.
\]

\end{examples}

\medskip

We conclude this paper by stressing the fact that Theorem \ref{gat} and Corollary \ref{cor0} apply only to {\it unramified} quadratic coverings of locally noetherian schemes (since every Galois covering is \'etale and, therefore, unramified). The behavior of the (cohomological) Brauer group functor under flat and {\it ramified} quadratic coverings has been discussed by a number of authors, at least in some particular cases. For instance, Ford \cite{f92,f17} and Skorobogatov \cite{sk} have discussed aspects of the indicated problem for certain types of algebraic varieties defined over an algebraically (or separably) closed field of characteristic different from $2$, using methods that are essentially different from those of this paper. In such investigations, quadratic {\it Galois} coverings sometimes play an important intermediate or auxiliary r\^{o}le (see, e.g., \cite[Theorems 2.1 and 2.2]{f92}). This suggests that the results of this section can be useful in future investigations on the behavior of the Brauer group functor under ramified quadratic coverings of {\it arbitrary} locally noetherian schemes.

\end{document}